\documentclass[psamsfonts]{amsart}

\usepackage[english]{babel}
\usepackage{amssymb}
\usepackage{url}

\theoremstyle{plain}
\newtheorem{theorem}{Theorem}

\newtheorem{lemma}[theorem]{Lemma}
\newtheorem{corollary}[theorem]{Corollary}
\newtheorem{proposition}[theorem]{Proposition}

\theoremstyle{definition}
\newtheorem{definition}[theorem]{Definition}

\theoremstyle{remark}
\newtheorem{rem}{Remark}[section]

\def\l{{\underline {l}}}

\def\N{\mathbb{N}}
\def\Z{\mathbb{Z}}
\def\Q{\mathbb{Q}}
\def\R{\mathbb{R}}
\def\C{\mathbb{C}}
\def\H{\mathbb{H}}
\def\P{\mathbb{P}}

\def\CC{\mathcal{C}}
\def\EE{\mathcal{E}}

\def\MM{\mathcal{M}}

\def\TT{\mathcal{T}}

\def\OO{\mathcal{O}}

\def\HD{\mathrm{HD}}
\def\diam{\mathrm{diam}}

\DeclareMathOperator{\sys}{sys}

\DeclareMathOperator{\SL}{SL}
\DeclareMathOperator{\Sp}{Sp}
\DeclareMathOperator{\Area}{Area}
\DeclareMathOperator{\dist}{dist}

\DeclareMathOperator{\Gal}{Gal}
\DeclareMathOperator{\MCG}{MCG}

\DeclareMathOperator{\SO}{SO}

\def\Re{\operatorname{Re}}
\def\Im{\operatorname{Im}}

\def\id{\mathrm{{id}}}

\def\epsilon{\varepsilon}

\begin{document}

\title{Weak mixing directions in non-arithmetic Veech surfaces}

\author{Artur Avila}

\address{CNRS UMR 7586, Institut de Math\'ematiques de Jussieu-Paris Rive
Gauche, B\^atiment Sophie Germain, Case 7012, 75205 Paris Cedex 13, France
\& IMPA, Estrada Dona Castorina 110, 22460-320, Rio de Janeiro, Brazil}
\email{artur@math.jussieu.fr}

\author{Vincent Delecroix}
\address{CNRS UMR 7586, Instittut de Math\'ematiques de Jussieu-Paris Rive
Gauche, B\^atiment Sophie Germain, 75205 Paris Cedex 13, France}
\curraddr{
LaBRI, UMR 5800, Bâtiment A30, 351, cours de la Libération
33405 Talence cedex, France.
}
\email{delecroix@math.jussieu.fr}


\begin{abstract}
We show that the billiard in a regular polygon is weak mixing in almost
every invariant surface, except in the trivial cases which give rise to
lattices in the plane (triangle, square and hexagon).
More generally, we study the problem of
prevalence of weak mixing for the directional flow
in an arbitrary non-arithmetic Veech surface, and
show that the Hausdorff dimension of the set of non-weak mixing directions
is not full.  We also provide a necessary condition, verified for instance
by the Veech surface corresponding to the billiard in the pentagon, for
the set of non-weak mixing directions to have positive
Hausdorff dimension.
\end{abstract}

\maketitle

\tableofcontents

\section{Introduction}

\subsection{Weak-mixing directions for billiards in regular polygons}
Let $n \geq 3$ be an integer and consider the billiard in an $n$-sided
regular polygon $P_n$.
It is readily seen that the $3$-dimensional phase space (the unit tangent
bundle $T^1 P_n$) decomposes into a one-parameter family of invariant
surfaces, as there is a clear integral of motion.
%\footnote {Taking for simplicity
%$P_n$ parallel to the horizontal axis, and writing
%$T^1 P_n=P_n \times S^1$, one sees that $\phi(p,z)=\Re z^m$ is constant on
%each trajectory, with $m=n$ or $[\frac {n} {2}]$ according to whether $n$ is
%odd or even.}
In such a setting, it is thus natural to try to understand the
dynamics restricted to each of, or at least most of,
the invariant surfaces.

The cases $n=3,4,6$ are simple to analyze, essentially
because they correspond to a lattice tiling of the plane: the dynamics is
given by a linear flow on a torus, so for a countable set of surfaces
all trajectories\footnote {Here we restrict consideration to orbits that do
not end in a singularity (i.e., a corner
of the billiard table) in finite time.} are periodic, and for all
others the flow is
quasiperiodic and all trajectories are equidistributed
with respect to Lebesgue measure on the surface.

For $n \neq 3,4,6$,
the invariant surfaces have higher genus, and
quasiperiodicity is not automatic anymore.  However,
W. Veech \cite{Veech1989} showed that a dichotomy still
holds: for a countable set of surfaces all infinite trajectories are periodic,
and for all others all trajectories
are equidistributed with respect to Lebesgue measure.

One important question that follows Veech work is whether weak mixing takes
place (absence of mixing is a general property in the more general class of
translation flows, which is known
from earlier work of A. Katok \cite{Katok1980}).
Recall that weak mixing means that there is no
reamainder of periodicity or quasiperiodicity from the measurable point of
view (i.e., there is no factor which is periodic or quasiperiodic), and
can thus be interpreted as the complete breakdown of the nice
lattice behavior seen for $n=3,4,6$.

While results about the prevalence of weak mixing were
obtained in the more general context of translation flows (\cite{AvilaForni2007}, \cite
{AF2}), the case of regular polygons proved to be much more resistent.
The basic reason is that the most succesful approaches so far were
dependent on the
presence of a suitable number of parameters which can be used in a
probabilistic exclusion
argument, and as a consequence they were not adapted to study the rigid
situation of a specific billiard table.  In this paper we address directly
the problem of weak mixing for exceptionally symmetric
translation flows, which include the ones arising from regular polygonal
billiards.

\begin{theorem}

If $n \neq 3,4,6$ then the restriction of the billiard flow in $P_n$ to almost every
invariant surface is weak mixing.

\end{theorem}

Of course, in view of Veech's remarkably precise answer to the problem of
equidistribution, one could wonder whether weak mixing is not only
a prevalent property, but one that might hold outside a countable set of
exceptions.  It turns out that this is not the case
in general, and in fact we will show that
the set of exceptions can be relatively large and have positive Hausdorff
dimension. However, we will prove that it can never have full dimension.

\subsection{Non-arithmetic Veech surfaces}

We now turn to the general framework in which the previous result fits.  A
\emph{translation surface} is a compact surface $S$
which is equipped with an atlas defined on the complement of a finite and
non-empty set of ``marked points'' $\Sigma$, such that the coordinate
changes are translations in $\R^2$ and each marked point $p$
has a punctured neighborhood isomorphic to a finite cover of a punctured disk
in $\R^2$.
The geodesic flow in any translation
surface has an obvious integral of motion, given by the angle of the
direction, which decomposes it into separate dynamical systems, the
\emph{directional flows}.

An \emph{affine homeomorphism} of a
translation surface $(S,\Sigma)$ is a homeomorphism of $S$ which fixes $\Sigma$
pointwise and which is affine and orientation preserving in the charts.
The linear part of such homeomorphism is well
defined in $\SL(2,\R)$. This allows one to define a homomorphism from the group of 
affine homeomorphisms to $\SL(2,\R)$: its image is a discrete subgroup
called the \emph{Veech group} of the translation surface.

A \emph{Veech surface} is an ``exceptionally symmetric'' translation surface,
in the sense that the Veech group is a (finite co-volume)
lattice in $\SL(2,\R)$ (it is easily seen that
the Veech group is never co-compact).  Simple
examples of Veech surfaces are \emph{square-tiled surfaces},
obtained by gluing finitely many copies of
the unit square $[0,1]^2$ along their sides. In this case the Veech group is
commensurable with $\SL(2,\Z)$.  Veech surfaces that can be derived
from square-tiled ones by an affine diffeomorphism are called \emph{arithmetic}.
By construction, arithmetic Veech surfaces are branched covers of flat tori, so
their directional flows are never topologically weak mixing (they admit
a \emph{continuous} quasiperiodic factor in \emph{any} direction).

The first examples of non-arithmetic Veech surfaces were described by
Veech, and correspond precisely to billiard flows on
regular polygons.  It
is easy to see that the billiard flow in $P_n$ corresponds, up to finite cover, to the geodesic
flow on a translation surface obtained by gluing the opposite sides of $P_n$
(when $n$ is even), or the corresponding sides of $P_n$ and $-P_n$ (when $n$
is odd).  This construction yields indeed a Veech surface $S_n$ which is
non-arithmetic precisely when $n \neq 3,4,6$.  The genus $g$ of $S_n$ is
related to $n$ by and $g=[\frac {n} {4}]$ ($n$ even) and $g=\frac {n-1} {2}$ ($n$ odd).

We can now state the main result of this paper:

\begin{theorem} \label{thm:generic_weak_mixing}

The geodesic flow in a non-arithmetic Veech surface is weakly mixing in
almost every direction.  Indeed the Hausdorff dimension of the
set of exceptional directions is less than one.

\end{theorem}

An important algebraic object associated to the Veech group $\Gamma$ of a Veech
surface $(S,\Sigma)$ is the
\emph{trace field} $k$ which is the field extension of $\Q$ generated
by the traces of the elements of
$\Gamma$.  Its degree $r=[k\mathop:\Q]$
satisfies $1 \leq r \leq g$ where $g$ is the genus of
$S$. Moreover, we have $r=1$ if and only if $S$ is arithmetic.
As an example, the trace field of $S_n$ is $\Q[\cos \frac {\pi} {n}]$
if $n$ is odd or $\Q[\cos \frac {2 \pi} {n}]$
if $n$ is even.%\textcolor {red} {Corrected formula for even.}

\begin{theorem} \label{thm:non_wm_expl_for_quad_trace_field}

Let $S$ be a Veech surface with a quadratic trace field (i.e., $r=2$).
Then the set of
directions for which the directional flow is not even
topologically weak mixing has positive Hausdorff dimension.

\end{theorem}

Notice that this covers the case of certain polygonal billiards
($\Q[\cos \frac {\pi} {n}]$ is quadratic if and only if $n \in \{4,5,6\}$, 
hence the above result holds for $S_n$ with $n \in \{5,8,10,12\}$),
and of all non-arithmetic
Veech surfaces in genus $2$.  We point out that
Theorem~\ref{thm:non_wm_expl_for_quad_trace_field} is a particular case
of a more
general result, Theorem~\ref{thm:non_wm_expl}, which does cover
some non-arithmetic Veech surfaces with non-quadratic trace fields
(indeed it applies to all non-arithmetic $S_n$ with $n \leq 16$,
the degrees of their trace fields ranging from $2$ to $6$)
and could possibly apply to all non-arithmetic Veech surfaces.  Let us also
note that the non-weak mixing directions obtained in
Theorem \ref {thm:non_wm_expl} have multiple rationally
independent eigenvalues.

One crucial aspect of our analysis is a better description of the
possible eigenvalues (in any minimal direction) in a
Veech surface.  Using
the algebraic nature of Veech surfaces, we are able to
conclude several restrictions on the group of eigenvalues
.
For non-arithmetic Veech surfaces,
the ratio of two eigenvalues always belong to $k$ and
the number of rationally independent eigenvalues is always at most
$[k\mathop:\Q]$.  Moreover,
the group of eigenvalues is finitely generated (so the \emph
{Kronecker factor} is always a minimal translation of a finite
dimensional torus).
Whereas for the case of
arithmetic Veech surfaces, we obtain that all
eigenvalues come from a ramified cover of a torus.
See Theorem~\ref{thm:nb_eig} for the precise statement.

We expect that, for a non-arithmetic Veech surface and along any minimal
direction that is not weak mixing, there are always exactly
$[k \mathop: \Q]$ independent eigenvalues, and that they are either all
continuous or all discontinuous.  This is the case along directions for
which the corresponding forward Teichm\"uller geodesic is bounded in
moduli space, see Remark~\ref{rmk:nb_eig_lin_rec}.

\subsection{Further comments}

The strategy to prove weak mixing for a directional flow on a translation
surface $S$ is to show that the associated unitary
flow has no non-trivial eigenvalues.
It is convenient to first rotate the surface so that the
directional flow goes along the vertical direction.
The small scale behavior of
eigenfunctions associated to a possible eigenvalue can then be studied using
renormalization.
Technically, one parametrizes all possible eigenvalues by
the line in $H^1(S;\R)$ through the imaginary part of the tautological one
form (the Abelian differential corresponding to the translation
structure) and then applies the
so-called Kontsevich-Zorich cocycle over the Teichm\"uller flow.
According to the Veech criterion any actual eigenvalue is parametrized by an
element of the ``weak-stable lamination'' associated to an
acceleration of the Kontsevich-Zorich cocycle acting modulo $H^1(S;\Z)$.
Intuitively, eigenfunctions parametrized by an eigenvalue outside the weak
stable lamination would exhibit so much oscillation in small
scales that measurability must be violated.
The core of \cite{AvilaForni2007} is a probabilistic method to
exclude non-trivial intersections of an arbitrary
fixed line in $H^1(S;\R)$ with the weak stable lamination, which uses basic
information on the non-degeneracy of the cocycle.

The problem of weak mixing in the case of $S_5$ was asked by C. McMullen 
during a talk of the first author on \cite{AvilaForni2007} in 2004.
It was realized during discussions with P. Hubert that the probabilistic
method behind \cite{AvilaForni2007} fails for Veech surfaces,
due essentially to a degeneracy
(non-twisting) of the Kontsevich-Zorich cocycle.  Attempts to improve the
probabilistic argument using Diophantine properties of invariant subspaces
turned out to lead to too weak estimates.

In this paper we prove that the locus of possible eigenvalues is
much more constrained in the case of Veech surfaces: eigenvalues must be
parametrized by an element in the ``strong stable lamination'',
consisting of all the
integer translates of the stable space, which is a much simpler object than
the weak stable lamination (considered modulo the strong stable space, the
former is countable, while the latter is typically uncountable).
Direct geometric estimates on the locus of intersection
can be then obtained using much less information on the non-degeneracy
of the cocycle.

In order to obtain this stronger constraint on the locus of possible
eigenvalues we will first carry out an analysis of the
associated eigenfunctions at scales corresponding to renormalizations
belonging to a large compact part of the moduli space (this refines Veech's
criterion, which handles compact sets that are small enough to
be represented in spaces of interval exchange transformations).  This is
followed by a detailed analysis of the excursion to the non-compact part of
the moduli space, which is used to forbid the occurrence of an integer
``jump'' in cohomology along such an excursion.  In doing so, we use
in an essential way the particularly simple nature of the
renormalization dynamics in the non-compact part of the moduli space
$\SL(2,\R)/\Gamma$ of a Veech surface (a finite union of cusps).

The fact that the flow is weakly-mixing in almost every direction
then follows from the absence of atoms in the harmonic measure of the
Kontsevich-Zorich cocycle. To get further and
obtain a bound on Hausdorff dimension of exceptional directions, we use
a Markov coding and quantitative estimates coming from a large deviation
upper bound in Oseledets theorem.

We should point out that it is reasonable to expect that, in the general
case of translation flows, one can construct examples of eigenvalues
which do not come
from the strong stable lamination.  Indeed M. Guenais and F. Parreau
construct suspension flows over ergodic
interval exchange transformations and with piecewise constant roof
function admitting infinitely many independent
eigenvalues (see Theorem~3 of~\cite{GP}), and
this provides many eigenvalues that do not come from the
strong stable lamination (which can only be responsible for a subgroup of
eigenvalues of finite rank).  See also \cite {BDM2}, section 6,
for a different example in a related context.
%\textcolor {red} {We must definitely check this with Hubert and Bressaud.}

{\bf Acknowledgements:} This work was partially supported by the ERC
Starting Grant ``Quasiperiodic'' and by the Balzan project of Jacob Palis. 
We would like to thank Pascal Hubert for several discussions and
Jean-Paul Thouvenot for his question about the number of eigenvalues
and the reference~\cite{GP}.

% 1 - introduction

\section{Preliminaries}

\subsection{Translation surfaces, moduli space, $\SL(2,\R)$-action} \label{subsec:tr_surf}
A \emph{translation surface} can also be defined as a
triple $(S,\Sigma,\omega)$ of a closed surface $S$, a non-empty finite set
$\Sigma \subset S$,
and an Abelian differential $\omega$ on $S$ whose zeros are contained in
$\Sigma$ ($\omega$ is holomorphic
for a unique complex structure on $S$)\footnote{A simpler definition
would be to consider $\Sigma$ to be the set of zeros of $\omega$. But in
that case we exclude the flat tori $\C / \Lambda$ with the Abelian differential
$dz$ which has no zero. Except in the special case of the torus, one can safely
take $\Sigma$ to be the zeros of $\omega$.}.  Writing
in local coordinates
$\omega = dz$, we get canonical charts to
$\C$ such that transition maps
are translations. Such a map exists at $x \in S$ if and only if $\omega$ is
non zero at $x$.  The zeros of $\omega$ are the \emph{singularities} of
the translation surface. Conversly, a translation surface $S$ defined as
in the introduction (in terms of a suitable atlas on $S \setminus \Sigma$)
allows one to recover the Abelian differential $\omega$ by declaring that
$\omega=dz$ and extending it uniquely to the marked points.
We write $(S,\omega)$ for the translation surface for which $\Sigma$ is the set
of zeros of $\omega$.
%We will often assume that $\omega$
%vanishes on $\Sigma$, so that all marked points are singularities (of
%course, the introduction of non-singular marked points does not change the
%dynamics of the directional flows we will define in a moment),
%and we will correspondingly denote the
%translation surface as a pair $(S,\omega)$.

Let $(S,\Sigma,\omega)$ be a translation surface. The form $|\omega|$ defines
a flat metric except at the singularities. For each
$\theta \in \R / 2\pi\Z$ we define the \emph{directional flow} in the
direction $\theta$ as the flow $\phi^{S,\theta}_T: S \rightarrow S$
obtained by integration of the unique vector field $X_\theta$ such
that $\omega(X_\theta) = e^{i \theta}$.  In local charts $z$ such that
$\omega = dz$ we have $\phi^{S,\theta}_T(z) = z + T e^{i \theta}$ for small
$T$. The directional flows are also called \emph{translation flows}.
The \emph{(vertical) flow} of $(S,\omega)$ is the flow
$\phi^{S,\pi/2}$
in the vertical direction. The flow is not defined at the zeros of
$\omega$ and hence not defined for all positive
times on backward
orbits of the singularities. The flows $\phi^{S,\theta}_T$
preserve the volume
form $\frac{1}{2i} \omega \wedge \overline{\omega}$ and the ergodic
properties of translation
flows that we discuss below refer to this measure.

Translation surfaces were introduced to study rational billiards as the
example of the regular polygons $P_n$ in the introduction. Each rational
billiard may be seen as a translation surface by a well known construction
called unfolding or Zemliakov-Katok construction
(see \cite{MasurTabachnikov2002}).

The following results hold for arbitrary translation surfaces:
the directional flow is minimal except for a countable set of directions
\cite{Keane1975}, the translation flow is uniquely ergodic except for a
set of directions of Hausdorff dimension at most $1/2$
\cite{KerckhoffMasurSmillie1986}, \cite{Masur1992}, and the translation flow
is not mixing in any direction \cite {Katok1980}.

The weak mixing property is more subtle.  Indeed, translation flows in a
genus one translation surface are never weakly mixing.  The same property
holds for branched coverings of genus one translation surfaces which
form a \emph {dense} subset of translation surfaces.  However, for
\emph {almost every}
translation surface of genus at least two, the translation flow
is weakly mixing in almost every direction \cite{AvilaForni2007}.
The implicit topological and measure-theoretical notions above refer to
a natural structure on the space of
translation structure that we introduce now.

Let $g,s \geq 1$ and let $S$ be a closed surface of genus $g$, let
$\Sigma \subset S$ be a subset with $\# S=s$ and let
$\kappa = (\kappa_x)_{x \in \Sigma}$ be a family of non-negative integers
such that $\sum \kappa_i = 2g-2$. The set of translation structures on
$S$ with prescribed conical angle $(\kappa_x + 1)2\pi$ at $x$, modulo
isotopy relative to $\Sigma$ forms a manifold $\TT_{S,\Sigma}(\kappa)$.
The manifold structure on $\TT_{S,\Sigma} (\kappa)$ is described by the
so-called \emph{period map} $\Theta: \TT_{S,\Sigma}(\kappa) \rightarrow
H^1(S,\Sigma; \C)$ which associates to $\omega$ its cohomology class in
$H^1(S,\Sigma; \C)$. The period map is locally bijective and provides
natural charts to $\TT_{S,\Sigma}(\kappa)$ as well as an
affine structure and a canonical Lebesgue measure.
We let $\TT^{(1)}_{S,\Sigma}(\kappa)$ denote the hypersurface of area $1$
translation surfaces.

The group $\SL(2,\R)$ acts (on left) on $\TT_{S,\Sigma}(\kappa)$ by postcomposition
on the charts and preserves the hypersurface
$\TT^{(1)}_{S,\Sigma}(\kappa)$. The subgroup of rotations
$r_\theta = \begin{pmatrix}\cos(\theta) & -\sin(\theta) \\ \sin(\theta)
& \cos(\theta) \end{pmatrix}$ acts by
multiplication by $e^{i \theta}$ on $\omega$.
The action of the diagonal subgroup
$g_t = \begin{pmatrix}e^t&0\\0&e^{-t}\end{pmatrix}$ is called the
\emph{Teichm\"uller flow} and transforms
$\omega = \Re(\omega) + i \Im(\omega)$ into
$g_t \cdot \omega = e^t \Re(\omega) + i e^{-t} \Im(\omega)$.
The \emph{stable} and \emph{unstable horocycle flows} are the action
of the matrices $\displaystyle h^-_s = \begin{pmatrix}1&0\\s&1\end{pmatrix}$
and $h_s^+ = \begin{pmatrix}1&s\\0&1\end{pmatrix}$.

The \emph{modular group} $\MCG(S,\Sigma)$ of $(S,\Sigma)$ is the group of
diffeomorphisms of $S$ fixing $\Sigma$ pointwise modulo isotopy relative to
$\Sigma$. It acts discretely discontinuously (on right) on the spaces
$\TT_{S,\Sigma}(\kappa)$ and $\TT^{(1)}_{S,\Sigma}(\kappa)$ via
$(S,\omega) \mapsto (S,\omega \circ d \phi)$. 
The quotient, denoted $\MM_{S,\Sigma}(\kappa)$ or
$\MM^{(1)}_{S,\Sigma}(\kappa)$ is
called a \emph{stratum of the moduli space of translation
surfaces of genus $g$ and $s$ marked points} or shortly a \emph{stratum}
\footnote{Our definition of strata are not standard as we consider a marking. It
matches the standard one up to a finite cover.}
The space
$\MM_{S,\Sigma}(\kappa)$ inherits from $\TT_{S,\Sigma}(\kappa)$ a complex
affine orbifold structure. 
The $\SL(2,\R)$ and $\MCG(S,\Sigma)$ actions on $\TT_{S,\Sigma}(\kappa)$ commute.
Hence the $\SL(2,\R)$ action is well defined on the quotient $\MM_{S,\Sigma}(\kappa)$.
The Lebesgue measure projects on
$\MM_{S,\Sigma}(\kappa)$ (respectively $\MM_{S,\Sigma}^{(1)}(\kappa)$) into
a measure
$\nu$ (resp. $\nu^{(1)}$) called the Masur-Veech measure.
Masur (\cite{Masur1982}) and Veech (\cite{Veech1982}) proved
independently that the measure $\nu^{(1)}$ has finite total mass, 
the action of $\SL(2,\R)$ on each
$\MM_{S,\Sigma}(\kappa)$ is measure preserving
and moreover the Teichm\"uller flow $g_t$ is ergodic on each
connected component of $\MM^{(1)}_{S,\Sigma}(\kappa)$. The moduli space, the Teichm\"uller flow and the
Kontsevich-Zorich cocycle are of main importance
in the results we mentioned above about the ergodic
properties of translation flows.

We will also sometimes
use the notation $\MM_g(\kappa)$ to denote
$\MM_{S,\Sigma}(\kappa)$, where the $(\kappa_j)_{1 \leq j \leq s}$ is
obtained by reordering the $(\kappa_x)_{x \in S}$ in non-increasing order.
As an example, for $n$ even, the surface $S_n$ built from a regular $n$-gon belongs
to the stratum $\MM_{\lfloor n/4 \rfloor}((n-4)/2)$ if $n \equiv 0 \mod 4$ 
or $\MM_{\lfloor n/4 \rfloor}( (n-6)/4, (n-6)/4)$ 
if $n \equiv 2 \mod 4$.

\subsection{Veech surfaces} \label{section:Veech_surfaces}
% Kenyon-Smillie 2000: the trace field of a Veech group is Q[lambda + lambda^-1] for any dilatation lambda of a pseudo-Anosov
% Kenyon-Smillie 2000: the trace field coincide with the holonomy field
% (see also Gutkin-Judge 2000 where they defined cross-ratio field and proved that for Veech surfaces cross-ratio field=trace field)
Our goal in this article is to study the weak-mixing property for
the directional flows in a translation surface with closed
$\SL(2,\R)$-orbits in $\MM_{S,\Sigma}(\kappa)$, which
are called Veech surfaces.

Let us recall that an \emph{affine homeomorphism} of a translation
surface $(S,\Sigma,\omega)$ is a homeomorphism of $S$ which preserves
$\Sigma$ pointwise and is affine
in the charts of $S$ compatible with the translation structure.
An affine homeomorphism $\phi$ has a well defined linear part, denoted by
$d \phi \in \SL(2,\R)$, which is the derivative of the action of $\phi$ in
the charts.  The set of linear
parts of affine diffeomorphisms forms a discrete
subgroup $\Gamma(S,\Sigma,\omega)$ of $\SL(2,\R)$ called
the \emph{Veech group} of $(S,\Sigma,\omega)$. A translation surface is
called a \emph{Veech surface} if its Veech group is a lattice.
The $\SL(2,\R)$ orbit of a Veech surface
is closed in $\MM_{S,\Sigma}(\kappa)$ and is naturally identified with
the quotient $\CC = \SL(2,\R) / \Gamma(S,\Sigma,\omega)$.
The $\SL(2,\R)$ action
on $\CC$ preserves the natural Haar
measure and the Teichm\"uller flow $g_t$ is the geodesic flow on the
unit tangent bundle of the hyperbolic surface $\H / \Gamma(S,\Sigma,\omega)$.

Veech proved that the Veech group of a Veech surface is
never co-compact. Moreover, the cusp excursions may be measured in terms
of systoles as in the well known case
of lattices with $\Gamma = \SL(2,\Z)$. 
A \emph{saddle connection} in a translation surface $(S,\Sigma,\omega)$
is a geodesic segment for the metric $|\omega|$ that start and ends
in $\Sigma$ and whose interior is disjoint from $\Sigma$.
For the square torus, $\C / (\Z 1 + \Z i)$ with the $1$-form $dz$ the set
of saddle connections identifies with primitive vectors (i.e. vector of the form
$pi + q$ with $p$ and $q$ relatively prime integers).
For a translation surface
$(S,\omega)$ the \emph{systole} $\sys(S,\omega)$ of $(S,\omega)$
is the length of the shortest saddle connection in $(S,\omega)$.
Assume that $(S,\omega)$ is a Veech surface and denote $\CC$ its $\SL(2,\R)$-orbit
in $\MM_{S,\Sigma}(\kappa)$.
Then the set $\CC_\epsilon =
\{(S,\omega) \in \CC:\ \sys(S,\omega) \geq \epsilon\}$
forms an exhaustion of $\CC$ by compact sets.

Beyond arithmetic surfaces (covers of the torus ramified over one point)
the first examples of Veech surfaces are the translation surfaces $S_n$
associated to the billiard in the regular polygon with $n$ sides $P_n
\subset \R^2$. The surface $S_n$ is built from $P_n$ ($n$ even)
or from the disjoint union of $P_n$ and $-P_n$
($n$ odd)~\cite{Veech1989}.  In either case, $S_n$ is defined by
identifying every side of $P_n$ with the unique
other side (of either $P_n$ or $-P_n$ according to the parity of $n$)
parallel to it, via translations. 
For them, the Veech group as well
as the trace field was computed by Veech.
\begin{theorem}[\cite{Veech1989}]
Let $S_n$ be the Veech surface associated to the regular polygon with $n$ sides. Then
  \begin{enumerate}
	\item if $n$ is odd, the Veech group of $S_n$ is the triangle group
$\Delta(2,n,\infty)$ with trace field $\Q[\cos(\pi/n)]$;
	\item if $n$ is even, the Veech group of $S_n$ is the triangle group
$\Delta(n/2,\infty,\infty)$ with trace field $\Q[\cos(2\pi/n))]$.
  \end{enumerate}
\end{theorem}
Notice that for $n$ even, $\Delta(n/2,\infty,\infty)$ is a subgroup
of index $2$ of $\Delta(2,n,\infty)$.
Other examples with triangle groups as Veech groups were discovered by
Ward~\cite{Ward1998} and Bouw-M\"oller~\cite{BouwMoeller2010}. Some of them are
obtained from billiards and a general construction with polygon gluings is
provided in~\cite{Hooper2012}. See also the article of Wright~\cite{Wright2013}
which explains these constructions in terms of Schwarz triangle maps.

Beyond triangle groups, Calta~\cite{Calta2004} and McMullen~\cite{McMullen2003} 
introduced an infinite family of Veech surfaces with quadratic trace field in the stratum
$\MM_2(2)$ (the family include $S_5$ and $S_8$). Later on, McMullen~\cite{McMullen2007} proved
that this is the complete list of non-arithmetic surfaces in $\MM_2(2)$
 and that in $\MM_2(1,1)$ the surface $S_{10}$ is the unique
non-arithmetic Veech surface. Other surfaces with quadratic trace
fields were discovered by McMullen~\cite{McMullen2006} in genus $3$ and $4$ and
further studied by Lanneau and Nguyen~\cite{LanneauNguyen2014}.

\subsection{Translation flow of Veech surfaces} \label{subsec:tr_flow_Veech}
Let $(S,\Sigma,\omega)$ be a translation surface and assume
that there exists an affine homeomorphism $\phi$ whose image
under the Veech group $g$ is parabolic.
The direction determined by the
eigenvector of $g$ in $\R^2$ is a \emph{completely periodic direction}
in $S$: the surface $(S,\Sigma,\omega)$ decomposes into a finite union of
cylinders whose waist curves are parallel to it.
Morever, $\phi$ preserves each cylinder and acts as a power of a
Dehn twist in each of them.
We may assume that the eigendirection is vertical, and hence
$g = h_s^-=\begin{pmatrix}1&0\\s&1\end{pmatrix}$ for some real number $s$.
Let $h_1$, $h_2$, \ldots, $h_k$ and $w_1$, $w_2$, \ldots, $w_k$ be the
widths and the heights of the cylinders $C_1$, \ldots, $C_k$ in the
vertical direction.
For each cylinder $C_i$, let $\phi_i$ be the Dehn twist in
$C_i$.  Its linear part is
$g_i = \begin{pmatrix}1&0\\\mu(C_i)^{-1}&1\end{pmatrix}$
where $\mu(C_i) = w_i / h_i$ is the \emph{modulus} of $C_i$ and
%Then, we have $\omega|_{C_i} \circ \phi_i = g_i \cdot \omega|_{C_i}$.
the real number $s$ is such that $s / \mu(C_i)$ are integers.
Reciprocally,
a completely periodic direction admits a non-trivial stabilizer
in $\SL(2,\R)$ if and only if the moduli $\mu(C_i)$ of the cylinders are
commensurable (their ratio are rational numbers). Such a direction is
called \emph{parabolic}.

Keane and Kerckhoff-Masur-Smillie theorems about minimality
and unique ergodicity of translation flows (see Section~\ref{subsec:tr_surf})
have the following refinement.
\begin{theorem}[Veech alternative, \cite{Veech1989}]  \label{thm:Veech_alternative}
  Let $(S,\Sigma,\omega)$ be a Veech surface. Then
  \begin{enumerate}
	\item either there exists a vertical saddle connection and
     the vertical direction in $(S,\Sigma,\omega)$ is parabolic;
	\item or the vertical flow is uniquely ergodic.
  \end{enumerate}
\end{theorem}

The $\SL(2,\R)$ orbit of a Veech surface is never compact (any fixed saddle
connection can be shrinked arbitrarily by means of the $\SL(2,\R)$ action,
thus escaping any compact subset of the moduli space).
Nevertheless, the geometry of flat surfaces in the cusps is well understood
and will be crucial in the study of eigenvalues 
(see Section~\ref{section:veech_criterion}).
If $\zeta$ and $\zeta'$ are two saddle
connections in a flat surface $(S,\omega)$ we
let $\zeta \wedge_{\omega} \zeta'$ denote the
number in $\C$ that corresponds to the (signed)
area of the parallelogram determined by $\omega(\zeta)$ and $\omega(\zeta')$.
\begin{theorem}[No small triangles] \label{thm:no_small_triangle}
  Let $(S,\omega)$ be a Veech surface. Then
there exists $\kappa > 0$ such that for any
pair of saddle connections $\zeta$ and $\zeta'$ 
\begin{itemize}
\item either $| \zeta \wedge_{\omega} \zeta'| > \kappa$;
\item or $\zeta$ and $\zeta'$ are parallel
(i.e. $\zeta \wedge_\omega \zeta' = 0$).
  \end{itemize}
\end{theorem}
The above theorem is actually the easy part of a
characterization of Veech surfaces proved in~\cite{SmillieWeiss2010}.
Note that the quantity $\zeta \wedge_{\omega} \zeta'$ is invariant
under the Teichm\"uller flow
(i.e. $\zeta \wedge_{g \cdot \omega} \zeta' = \zeta \wedge_\omega \zeta'$ for any
$g \in \SL(2,\R)$) and corresponds to twice the area of a (virtual)
triangle delimited by $\zeta$ and $\zeta'$. We deduce in particular,
that if there exists a small saddle connection in a Veech surface, then
any other short saddle connection would be parallel to it and that the
smallness is uniform for the whole $\SL(2,\R)$-orbit. More precisely,
\begin{corollary} \label{cor:only_one_family_of_small_geodesics}
Let $\CC$ be a $\SL(2,\R)$-orbit of a Veech surface.
Then there exists $\epsilon>0$
such that for any $u \not \in \CC_\epsilon$ the saddle connections
in $u$ shorter than $1$ are parallel to the direction of the
shortest saddle connection in $u$.
\end{corollary}

\subsection{Holonomy field and conjugates of Veech group} \label{subsec:hol_field_and_conj}
\label{section:Galois_conjugate}
Let $(S,\Sigma,\omega)$ be a translation
surface and let $\Lambda = \omega(H_1(S;\Z)) \subset \C \simeq \R^2$
be the module of periods of $\omega$. In what follows, periods will be sometimes
called holonomies.  Let $e_1$ and $e_2$ be two non-parallel
elements in $\Lambda$. The \emph{holonomy field} of $(S,\Sigma,\omega)$ is the
smallest field $k$ of $\R$ such that any element in $\Lambda$ may be
written as $k$-linear combination of $e_1$ and $e_2$. 
For any subgroup $\Gamma$ of $\SL(2,\R)$, the \emph{trace field} of $\Gamma$ is the field generated by the traces of the element of $\Gamma$.

\begin{theorem}[Gutkin-Judge \cite{GutkinJudge2000}, Kenyon-Smillie \cite{KenyonSmillie2000}]
Let $(S,\Sigma,\omega)$ be a Veech surface. Then its holonomy field
$k$ coincides with the trace field of its Veech group. The degree of
$k$ over $\Q$ is at most the genus of $S$ and the rank of
$\Lambda = \omega(H_1(S;\Z))$ is twice the degree of $k$ over $\Q$.
\end{theorem}

We will need two important facts about the holonomy field of a Veech surface.
\begin{theorem}[Gutkin-Judge \cite{GutkinJudge2000}]
A Veech surface $(S,\Sigma,\omega)$ is arithmetic (i.e. a ramified cover of a torus over one point)
if and only if its holonomy field is $\Q$.
\end{theorem}

\begin{theorem}
[\cite{LanneauHubert2006}] \label{thm:trace_field_totally_real}
The holonomy field of a Veech surface is totally real.\footnote {Recall that a
subfield $k \subset \R$ is called totally real if its image under
any embedding $k \rightarrow \C$ is contained in $\R$.}
\end{theorem}

Now, we define the Galois conjugates of the Veech group.
Let $(S,\Sigma,\omega)$ be a Veech surface, let $\Gamma$ be its
Veech group and let $k$ be its holonomy field.
Let $e_1$ and $e_2$ be two non-parallel elements in the set
of holonomies $\Lambda = \omega(H_1(S;\Z))$.
For each element $v \in H_1(S;\Z)$ there exist
unique elements $\alpha$ and $\beta$ of $k$ such that
$\omega(v) = \alpha e_1 + \beta e_2$.
The maps $\alpha$ and $\beta$ are linear with values in $k$, in other words
they belong to $H^1(S;k)$, and moreover,
the tautological space $V = \R \Re(\omega) \oplus \R \Im(\omega)$ can be
rewritten as $V=\R \alpha \oplus \R \beta$ in $H^1(S;\R)$.
Note that an alternative definition of the trace field would be the field of
definition of the plane $\R \Re(\omega) + \R \Im(\omega)$ in
$H^1(S; \R)$.
For any embedding $\sigma: k \rightarrow \R$, we may define new linear
forms $\sigma \circ \alpha$ and $\sigma \circ \beta$.  Those linear forms
generate a $2$ dimensional subspace in $H^1(S;\R)$.
The subspace does
not depend on the choice of $\alpha$ and $\beta$ and
we call it the \emph{conjugate by $\sigma$} of the
tautological subspace $V$.
This subspace (which is indeed defined in $H^1(S;k)$)
will be denoted by $V^\sigma$.
Because the action of the affine group on homology
is defined over $\Z$ and preserves $V$, it preserves as well the conjugates
$V^\sigma$. The following result appeared in~\cite{Moeller2006} in terms
of decomposition of variation of Hodge structures.
\begin{theorem}[\cite{Moeller2006}]
Let $(S,\Sigma,\omega)$ be a Veech surface,
$k$ its holonomy field and $V = \R \Re \omega \oplus \R \Im \omega$
be the tautological subspace. Then for any embedding
$\sigma: k \rightarrow \R$ the subspace $V^\sigma$ is invariant
under the action of the affine group of $(S,\Sigma,\omega)$. Moreover,
the space generated by the $[k\mathop:\Q]$ subspaces $V^\sigma \subset H^1(S;\R)$
is the smallest rational subspace of $H^1(S;\R)$ containing $V$ and
is the direct sum of the $V^\sigma$.
\end{theorem}
The fact that the sum is direct follows from the presence of
hyperbolic elements in the Veech group
(see Theorem~28 of \cite{KenyonSmillie2000}).

For an affine homeomorphism $\phi$ of a translation surface
$(S,\Sigma,\omega)$ we have
\[
\begin{pmatrix}\phi^* \Re(\omega)\\ \phi^* \Im(\omega))\end{pmatrix}
=
D\phi \cdot \begin{pmatrix} \Re(\omega) \\ \Im(\omega) \end{pmatrix}.
\]
Hence, the Veech group $\Gamma$ is canonically identified to the action
of the affine group on the tautological subspace
$V = \R \Re\omega \oplus \R \Im\omega$.
The choice of two elements of $H_1(S;\Z)$ with non-parallel holonomy
provides an identification of $\Gamma$ as a subgroup of $\SL(2,k)$.
Given an embedding of $k$ in $\R$ we may conjugate the coefficients
of the matrices in $\SL(2,k)$ and get a new embedding of $\Gamma$
into $\SL(2,\R)$. This embedding is canonically identified to the
action of the affine group on the conjugate of the tautological bundle
$V^\sigma$. We denote by $\Gamma^\sigma$ this group and call it
the \emph{conjugate of the Veech group} by $\sigma$.

\subsection{Kontsevich-Zorich cocycle} \label{subsec:KZ}
Over a Teichm\"uller space $\TT_{S,\Sigma}(\kappa)$, let us consider the
trivial cocycle $g_t \times id$ on $\TT_{S,\Sigma}(\kappa) \times
H^1(S;\R)$. The modular group
$\MCG(S,\Sigma)$ acts on $\TT_{S,\Sigma}(\kappa) \times H^1(S; \R)$ and the
quotient bundle is a flat orbifold vector bundle over
$\MM_{S,\Sigma}(\kappa)$ called the \emph{Hodge bundle}.
The \emph{Kontsevich-Zorich cocycle} is the projection of
$g_t \times id$ to the Hodge bundle. We will also need a
slightly different form of the Kontsevich-Zorich cocycle,
namely the projection of $g_t \times id$
on $\TT_{S,\Sigma}(\kappa) \times H^1(S \backslash \Sigma; \R)$
that we call the \emph{extended Kontsevich-Zorich cocycle} (on the \emph
{extended Hodge bundle}).

We will often use a discrete version of the Teichm\"uller flow and
the Kontsevich-Zorich cocycle built as follows. Let $Q$ be a transveral to the Teichm\"uller flow which is
compact and simply connected. There is a natural trivialization of the
Hodge bundle and extended Hodge bundle over $Q$ given by the Gauss-Manin connection
(there is a unique way to identify two nearby fibers preserving the integer
lattice). Once we fixed a reference based point $(S,\Sigma,\omega) \in Q$, the Poincar\'e
section $F:Q \rightarrow Q$ of the Teichm\"uller flow hence provides two cocycles
$A: Q \rightarrow \Sp(H^1(S;\R))$ and $A': Q \rightarrow \Sp(H^1(S \backslash \Sigma; \R))$ that
preserve the integer lattices.

In the case of Veech surfaces, the cocycles $A:Q \rightarrow \Sp(H^1(S;\R))$
and $A': Q \rightarrow \Sp(H^1(S \backslash \Sigma \Z; \R))$ preserve the tautological
bundle $V$ and its conjugates $V^{\sigma}$. As each homotopy class of closed curve
on a hyperbolic surface has either a geodesic or a horocyclic representative, one can
see that the matrices $A(x)$ which appear in this discrete version are induced
by affine homeomorphisms of the surface $(S,\Sigma,\omega)$.

Because the modular group
acts by symplectic transformation on $H^1(S;\R)$ (with respect to the intersection form)
the $2g$ Lyapunov exponents
$\lambda_1^\mu \geq \lambda_2^\mu \geq \ldots \geq \lambda_{2g}^\mu$ of the
Kontsevich-Zorich cocycle satisfy
\[
\forall 1 \leq k \leq g, \quad \lambda^\mu_k = -\lambda_{2g-k-1} \geq 0.
\]
Because of the natural injection
$H^1(S; \R) \rightarrow H^1(S \backslash \Sigma;\R)$,
the Lyapunov spectrum of the extended Kontsevich-Zorich cocycle
contains the one of the Kontsevich-Zorich cocycle. It may be proved
that the remaining exponents are $s-1$ zeros where $s$ is the
cardinality of $\Sigma$.

The Kontsevich-Zorich cocycle preserves the tautological bundle and
one sees directly from the definition that
the Lyapunov exponents on the tautological bundle
are $1$ and $-1$.  For the remaining exponents we have the
following inequality.
\begin{theorem}[Forni \cite{Forni2002}] \label{thm:Forni_inequality}
Let $\mu$ be an ergodic invariant probability measure of the
Teichm\"uller flow on some stratum $\MM_{S,\Sigma}(\kappa)$. Then the
second Lyapunov exponent $\lambda^\mu_2$ of the Kontsevich-Zorich
cocycle satisfies $1 > \lambda^\mu_2$.
\end{theorem}

\begin{rem} \label {bound}
Forni's proof indeed shows that there exists a natural Hodge norm on the
Hodge bundle such that for any $x \in
\MM_{S,\Sigma}(\kappa)$, the Kontsevich-Zorich cocycle starting from $x$
has norm strictly less than $e^t$ at any time $t>0$, when restricted to the
symplectic orthogonal to the tautological space. 
\end{rem}

Given a Veech surface $(S,\Sigma,\omega)$,
the conjugates of the tautological subspace $V^\sigma$ are well defined over the
whole Teichm\"uller curve: they form subbundles of the
Hodge bundle invariant for the Kontsevich-Zorich cocycle.
In other words, we may restrict the Kontsevich-Zorich cocycle to any
of the $V^\sigma$ and consider the associated pair $(\lambda^\sigma,-\lambda^\sigma)$
of Lyapunov exponents. The following result shows that these exponents are non-zero.
\begin{lemma}[\cite{BouwMoeller2010}] \label{lem:separation_LE_Veech_surfaces}
Let $(S,\Sigma,\omega)$ be a Veech surface, $k$ its holonomy field
and $V$ be the tautological subbundle of the Hodge bundle
of its $\SL(2,\R)$ orbit. Then for any non-identity embedding
$\sigma: k \rightarrow \R$, the non-negative Lyapunov exponent
$\lambda^\sigma$ of the Kontsevich-Zorich cocycle restricted to
$V^\sigma$ satisfies $0 < \lambda^\sigma < 1$.
\end{lemma}

Note that the upper bound can be deduced from Theorem~\ref{thm:Forni_inequality}.
The lower bound can also be obtained from the fact that the group $\Gamma^\sigma$ is non-elementary and
the hyperbolicity properties of the geodesic flow on $\CC$.

One consequence of the upper bound for $\lambda^\sigma$ is the following result
that will be used to prove Theorem~\ref{thm:non_wm_expl_for_quad_trace_field}.
\begin{proposition}[\cite{Moeller2006}] \label{prop:conjugates_indiscrete}
  Let $(S,\Sigma,\omega)$ be a Veech surface, $k$ its holonomy field and
  $\Gamma$ its Veech group. Then, for any non-identity embedding $\sigma: k
  \rightarrow \R$ the group $\Gamma^\sigma$ is non-discrete.
\end{proposition}

\section{Eigenfunctions in Veech surfaces} \label{section:veech_criterion}

Let $(S,\Sigma,\omega)$ be a translation surface and
$\phi_t:S \rightarrow S$ the vertical flow. We say that
$\nu \in \R$ is an \emph{eigenvalue} of $\omega$ if there exists
a non-zero measurable function
$f:S \rightarrow \C$ such that for almost every
$x \in S$,we have $f(\phi_T(x)) = \exp(2 \pi i\nu T) f(x)$ for all $T \in \R$.
%(note: a point is \emph{regular} if it has infinite future orbit).
The eigenvalue is \emph{continuous} if the map $f$ may be chosen
continuous.  The flow is \emph{weak-mixing} if it admits no
eigenvalue except $0$ with multiplicity one.

The Veech criterion that appeared in~\cite{Veech1984} was of main
importance in \cite{AvilaForni2007} to prove the genericity of
weak-mixing among translation flows.  This criterion depends on the
consideration of appropriate compact transversal to the Teichm\"uller
flow which is ``small enough'' to fit inside ``zippered rectangles'' charts
and also satisfy some additional boundedness properties.

\begin{theorem}[Veech criterion] \label{thm:Veech_criterion}
Let $\MM_g(\kappa)$ be a stratum of translation surfaces.
For all $(S,\Sigma,\omega)$ in $\MM_g(\kappa)$ there exists a small 
compact transversal $\Delta$ of the Teichm\"uller flow containing $(S,\Sigma,\omega)$
such that the following holds. 
Let $A_n$ for all $(S,\Sigma,\tau)$ in that transversal that is
recurrent and admits an eigenvalue $\nu$, the values of the Kontsevich-Zorich cocycle
$A_n(\tau)$ for that transversal satisfy
Let $A_n: \Delta \rightarrow \Sp(2g,\Z)$ be the
Kontsevich-Zorich cocycle on that transversal.
For every $x \in \Delta$ such that its forward Teichm\"uller geodesics comes back
infinitely often to $\Delta$, we have
\[
\lim_{n \to \infty} \dist(A_n(x) \cdot (\nu \Im (\tau_x)),
H^1(S \backslash \Sigma,\Z)) = 0,
\]
where $\tau_x$ is the differential associated to the surface $x \in \Delta$.
\end{theorem}

The above theorem holds in great generality as soon as the dynamical system
is described by Rokhlin towers (see \cite{BressaudDurandMaass}).

Veech's criterion tells us that we can prove that the vertical flow is weak
mixing if we can show that the line $\R \Im(\omega)$ intersects the
``weak stable lamination'', the set of all
$w \in H^1(S \backslash \Sigma; \R)$ such that
\[
\lim_{n \to \infty} \dist(A_n \cdot w,
H^1(S \backslash \Sigma,\Z)) = 0,
\]
only at the origin.  Unfortunately, the nature of the weak stable lamination
is rather complicated.  It is of course a union of translates of the
stable space, the set of all
$w \in H^1(S \backslash \Sigma; \R)$ such that
\[
\lim_{n \to \infty} \dist(A_n \cdot w,0) = 0,
\]
and it contains the ``strong stable space'' consisting of
the integer translates of the stable space. 
However, in general it is much larger, being transversely uncountable.

The main objective of this section is to
show that for Veech surfaces, any eigenvalue $\nu$ must be such
that $\nu \Im(\omega)$ must belong to the smaller (and much simpler)
strong stable space.

\begin{theorem} \label{thm:Veech_criterion_for_Veech_surfaces}
  Let $(S,\Sigma,\omega)$ be a Veech surface with no
vertical saddle connection
and whose linear flow admits an eigenvalue $\nu$. Consider a
compact transversal $\Delta$ large enough so that the forward Teichm\"uller
geodesic of $(S,\Sigma,\omega)$ comes back infinitely often to $\Delta$.

Let $A: \Delta \rightarrow \Sp(H^1(S \backslash \Sigma, \Z))$ the associated  
Kontsevich-Zorich cocycle. Then there exists
$v \in H^1(S \backslash \Sigma; \Z)$ such that
\[ \lim_{n \to \infty} A_n \cdot (\nu \Im(\omega) - v) = 0. \]
\end{theorem}
Note that the existence of such transversals is due to the fact that geodesics
in finite volume hyperbolic surfaces always come back to a fixed compact set.

In order to prove Theorem \ref{thm:Veech_criterion_for_Veech_surfaces},
we will need to control all the renormalizations of
eigenfunctions, and not only those corresponding to returns to a small
compact transversal.  It will be crucial for our strategy that
for a Veech surface we can choose $\epsilon$ such that the following holds.
Let $\CC_\epsilon=\{\omega:\, \sys(S,\Sigma,\omega) \geq \epsilon\}$. Then the set of ``moments of
compactness'' $\{t>0;\, g_t \cdot \omega \in \CC_\epsilon\}$ for the forward
Teichm\"uller geodesic is unbounded if and only if there are no vertical
saddle connections in $\omega$. And, most importantly, any orbit segment away from the
moments of compactness (the cusp excursions)
can be easily described geometrically.

\subsection{Tunneling curves and a dual Veech criterion}
In this section (which is not restricted to Veech surfaces)
we show that the existence of eigenfunctions yields
information about all times of the
Teichm\"uller flow and not only return times to a (small or large)
compact transversal.
This is based on a refinement of the Veech criterion (Theorem~\ref{thm:dual_Veech_criterion})
which is formulated in terms of homology cycles called
\emph{tunneling curves} (which are designed to follow closely the vertical
flow in a suitable sense). We show that tunneling curves always see
the expected property of approximation to integers.
In a second step
(Lemma~\ref{lem:nice_basis_in_compact_part}), we prove that in any compact part
of the moduli space, the tunneling curves generate $H_1(S \backslash \Sigma;
\Z)$.  Those two results together allow us to
remove the smallness condition on the  
transversal in the formulation of the Veech criterion, hence allowing us
to consider the large compact set $\CC_\epsilon$ when analysing Veech
surfaces later.

Before defining tunneling curves we need the notion of cycle of rectangles.
A \emph{rectangle} for $(S,\Sigma,\omega)$ is an isometric immersion of an
euclidean rectangle with horizontal and vertical sides.
In other words, a rectangle is a map
$R: [0,w] \times [0,h] \rightarrow S \setminus \Sigma$ such that
$R^* (\Re(\omega)) = \pm dx$ and $R^* (\Im(\omega)) = \pm dy$.
The number $w$ is called the \emph{width}
of the rectangle and the number $h$ its \emph{height}.
Note that with our convention we may not identify a
rectangle with its image in $S$, we care about the direction:
a rectangle is determined by its image in $S$ and an
element of $\{+1,-1\} \times \{+1, -1\}$.
\begin{definition}
  A $(k,\delta,h)$-\emph{cycle of rectangles} for $\omega$ is a set
of $2k$ rectangles denoted $H_j$ and $V_j$ for $j \in \Z/k \Z$ such that
\begin{itemize}
\item the height of $H_j$ is $\delta$ and its width is $w_j \geq \delta$;
\item the width of $V_j$ is $\delta$
and its height is $\delta \leq h_j \leq h$;
  \item $H_j^*(\Re(\omega)) = \pm dx$ and $H_j^*(\Im(\omega)) = dy$;
  \item $V_j^*(\Re(\omega)) = dx$ and $V_j^*(\Im(\omega)) = \pm dy$;
  \item each rectangle $H_j$ is embedded in the surface;
  \item for each $j$; 
	$H_j ([0,\delta] \times [0,\delta]) = V_{j-1}([0,\delta] \times [h_j-\delta,h_j])$
	and
	$H_j ([w_j-\delta,w_j] \times [0,\delta]) = V_j([0,\delta] \times [0,\delta])$.
\end{itemize}
\end{definition}
In other words, a $(k,\delta,h)$-cycle of rectangles is a thin tube
of width $\delta$ in $(S,\Sigma,\omega)$ made of $k$ horizontal and
$k$ vertical pieces and that forms a cycle in the surface. We will
sometimes drop the condition on the heights and write
$(k,\delta)$ for $(k,\delta,\infty)$.

A tunneling curve is a curve which belongs in a cycle of rectangles.
More precisely, let $R = (H_j,V_j)_{j \in \Z/ k \Z}$ be a
$(k,\delta,h)$-cycle of rectangles for $\omega$.
We may build a curve $\zeta$ as follows:
for $j \in \Z/k \Z$, we define vertical segments
$\zeta^v_j: [0,h_j-\delta] \rightarrow S \backslash \Sigma$ by
$\zeta_v^j(t) = V_j(\frac {\delta} {2}, t+\frac {\delta} {2})$
and horizontal segments
$\zeta^h_j: [0,w_j-\delta] \rightarrow S \backslash \Sigma$ by
$\zeta^h_j(t) = H_j( t+\frac {\delta} {2},\frac {\delta} {2})$.
The curve $\zeta$ is the concatenation of $\zeta^h_1$,
$\zeta^v_1$,..., $\zeta^h_k$, $\zeta^v_k$
and forms a loop in the surface $S \backslash \Sigma$.  The homology
class of $\zeta$ in $H_1(S \backslash \Sigma; \Z)$
is the \emph{homology class} of the cycle of rectangles $R$.

A homology class $\zeta \in H_1(S \backslash \Sigma; \Z)$ is said
to be $(k,\delta,h)$-\emph{tunneling} if there exists
a set of $(k_i,\delta,h)$-cycles of rectangles for
$i=1,\ldots,n$ such that $k_1 + \ldots + k_n \leq k$ and whose
homology classes $\zeta_i$ satisfy $\sum \zeta_i = \zeta$.

Note that if $\zeta$ is a
tunneling curve in a $(k,\delta,h)$-cycle of rectangles, then 
\[
|\Re(\omega)(\zeta)| \leq \int_\zeta |\Re(\omega)| \leq \frac{k}{\delta} \Area(\omega)
\quad \text{and} \quad
|\Im(\omega)(\zeta)| \leq \int_\zeta |\Im(\omega)| \leq k h.
\]
In particular, in a fixed translation surface $(S,\Sigma,\omega)$ there is
only a finite number of $(k,\delta,h)$-tunneling curves. The set of
$(k,\delta,h)$-tunneling
homology classes in $H_1(S \backslash \Sigma; \Z)$ for $\omega$ is denoted
$TC_{k,\delta,h}(\omega)$. The set $TC_{k,\delta}(\omega) =
TC_{k,\delta,\infty}(\omega)$ denotes the set of $(k,\delta)$-tunneling
homology classes. Note that, if $k' \leq k$, $\delta' \geq \delta$ and $h' \leq h$
then a $(k',\delta',h')$-tunneling curve is also $(k,\delta,h)$-tunneling.

We will now adapt Veech's original proof of his criterion in
Veech~\cite{Veech1984} to obtain a dual version with respect to
the tunneling curves in $TC_{k,\delta}$.

\begin{theorem}[dual Veech criterion] \label{thm:dual_Veech_criterion}
Let $(S,\Sigma,\omega)$ be a translation surface without vertical
saddle connections and that admits
an eigenvalue $\nu$. Then, for any
positive integer $k$ and positive real number $\delta$ we have
\[
\lim_{t \to \infty} \sup_{\zeta \in TC_{k,\delta}(g_t \cdot \omega)}
\dist(\nu \Im(\omega)(\zeta),\Z) = 0.
\]
\end{theorem}

\begin{proof}
We assume that $\Area(\omega) = 1$.

Fix $k$ and $\delta$. We fix a small number $\alpha$ and we prove that
for $t$ big enough, all $(k,\delta)$-tunneling curves for
$\omega_t = g_t \cdot \omega = e^t \Re(\omega) + I e^{-t} \Im(\omega)$
are such that $\dist (\nu \Im(\omega)(\zeta), \Z) < k \alpha$. It is
enough to prove the theorem for a curve that belongs to a
cycle of rectangles (recall that a tunneling curve may be a sum of
curves associated to cycle of rectangles).

Let $(H_j,V_j)_{j \in \Z/ k \Z}$ be a $(k,\delta)$-cycle of rectangles for
$\omega_t$ and let $\zeta$ be its homology class. We define
the \emph{signed height} of the vertical rectangle $V_j$ by
$\widetilde{h_j} = h_j - \delta$ if
$(V_j)^*(\Im(\omega)) = dy$ and
$\widetilde{h_j} = \delta -h_j$ otherwise (it is precisely the value of
the integral of $\Im(\omega)$ along the component $\zeta^v_j$ of
the curve $\zeta$ as above).
In particular,
we have $\Im(\omega)(\zeta)= \widetilde{h_1} + \ldots + \widetilde{h_k}$.
For each $j$, we define $I_j = H_j([0,w_j] \times \{\delta/2\})$ the middle
interval of the rectangle $H_j$. The segment $I_j$ is an horizontal
interval of length $e^{-t} w_j$ for $\omega$.

We assume that the surface $S$ admits a non-trivial eigenvalue
$\nu \not= 0$ and denote $f$ an associated eigenfunction $f:S \to \C$ with
$|f|=1$.  
The set $\Omega \subset S$ of points $x$ such that for all $T \in \R$
the vertical flow is defined and $f(\phi_T(x)) = e^{2 \pi i \nu T} f(x)$ has
full measure.
Moreover, $\Omega$ intersects any
horizontal segment in a subset of full linear measure.
Define measurable
functions $f_j:[0,w_j] \rightarrow \C$ by $f_j(t)=f(H_j(t,\delta/2))$.
For each $j$ and almost every $x \in [0,\delta]$ we have
\[
\begin{array}{ll}
  f_j(w_j-\delta+x) = e^{-2 \pi i \nu \widetilde{h_j}} f_{j+1}(x)
& \text{if $H_j^*(\Re(\omega)) = dx$ and $H_{j+1}^*(\Re(\omega)) = dx$}\\
  f_j(w_j-x) = e^{-2 \pi i \nu \widetilde{h_j}} f_{j+1}(x)
& \text{if $H_j^*(\Re(\omega)) = -dx$ and $H_{j+1}^*(\Re(\omega)) = dx$}\\
  f_j(w-\delta+x) = e^{-2 \pi i \nu \widetilde{h_j}} f_{j+1}(\delta-x)
& \text{if $H_j^*(\Re(\omega)) = dx$ and $H_{j+1}^*(\Re(\omega)) = -dx$} \\
  f_j(w_j-x) = e^{-2 \pi i \nu \widetilde{h_j}} f_{j+1}(\delta-x)
& \text{if $H_j^*(\Re(\omega)) = -dx$ and $H_{j+1}^*(\Re(\omega)) = -dx$}
\end{array}
\]
These formulas can be rewritten in more compact form as follows. Let $s_j:[0,\delta] \to
[0,\delta]$ be given by $s_j(x)=\delta-x$ if $H_j^*(\Re \omega)) = dx$ and
by $s_j(x)=x$ if $H_j^*(\Re \omega)=-dx$.
Then $f_j(w_j-s_j(x))=e^{-2 \pi i \nu
\widetilde {h_j}} f_{j+1}(\delta-s_{j+1}(x))$ in all cases.

Recall that we fix $\alpha$ an arbitrarily small positive real number.
The strategy of the proof, consists in proving that if $t$ is big enough,
independently of the choice of the $(k,\delta)$-cycle of rectangles,
we may find points $x_j \in [0,\delta]$, $j \in \Z/ k \Z$ such that
for each $j$ we have
$|f_j(\delta-s_j(x_{j-1}))-f_j(w_j-s_j(x_j))|<\alpha$.  In particular,
we can write $\frac {f_j(\delta-s_j(x_{j-1}))}
{f_j(w_j-s_j(x_j))}=e^{2 \pi i \lambda_j}$ where $\lambda_j \in
(-\alpha,\alpha)$.

Assuming that such points do exist, we prove how to derive our
theorem.  Using the points $x_j$ we may write
\[
\begin{array}{ll}
1&=\prod_{j \in \Z/ k \Z}
\frac {f_j(w_j-s_j(x_j))} {f_{j-1}(w_{j-1}-s_{j-1}(x_{j-1}))}=
\prod_{j \in \Z/ k \Z}
e^{2 \pi i \nu \widetilde {h_{j-1}}}
\frac {f_j(w_j-s_j(x_j))} {f_j(\delta-s_j(x_{j-1}))}\\
&= \prod_{j \in \Z/ k \Z} e^{2 \pi i \nu \widetilde {h_{j-1}}} e^{-2 \pi i
\lambda_j}
\end{array}
\]
so that $\Im(\omega)(\zeta)=\sum \lambda_j \mod \Z$, implying the result.

Now, we show how to find points $x_j$ using a measure theoretic
argument.  More precisely, we prove that the measure of
the set of points $(x,y) \in [0,\delta] \times [w_j-\delta, w_j]$
such that $|f_j(x) - f_j(y)| < \alpha$ becomes arbitrarily close to
$\delta^2$ as $t$ goes to infinity independently of the choice of the cycle
of rectangles.  Since $\delta \leq w_j \leq \delta^{-1}$,
it is enough to show that there exists
a compact subset $K_j \subset I_j$ with measure close to $1$ such that
$|f(x)-f(y)|<\alpha$ for every $x,y \in K_j$.

Fix some small constant $\chi>0$.  By Lusin's Theorem,
there exists a compact
subset $K \subset S$ of measure $1-\chi$ such that $f|K$ is continuous (up to a measure zero set). 
In particular, there exists $\epsilon>0$ such that if
$x,y \in K$ are $\epsilon$-close then $|f(x)-f(y)|<\alpha$.

Notice that the rectangle $H_j$ has width $e^{-t} w_j$ and height $e^t
\delta$ in $(S,\Sigma,\omega)$. Recall that we also have $\delta \leq w_j \leq \frac {1}
{\delta}$. In particular the area of $H_j$ is at least $\delta^2$, so
$K$ must intersect it in a subset of measure at least $1-\delta^{-2}
\kappa$.  It follows that some
(full) horizontal segment $I'_j=H_j([0,w_j] \times \{T\})$
in this rectangle intersects $K$ into a subset $K'_j$ of
measure at least $1-\delta^{-2} \chi$ as well.
Take $t$ large enough so that $e^{-t} \delta^{-1}<\epsilon$.  Then
$|f(x)-f(y)|<\alpha$ for every $x,y \in K'_j$.  Note that
$I'_j=\phi_{\pm e^t (T-\delta/2)}(I_j)$ (the same sign as when writing
$H_j^*(\Im \omega)=\pm dy$) so by the functional
equation we have $|f(x)-f(y)|<\alpha$
for every $x,y \in K_j=\phi_{\mp e^t (T-\delta/2)}(K'_j)$, as desired.
\end{proof}

We now prove that any
translation surface
admits tunneling basis
and the constant may be taken uniform in compact sets.

\begin{lemma} \label{lem:nice_basis_in_compact_part}
  Let $\MM^{(1)}_{S,\Sigma}(\kappa)$ be a stratum in moduli space
and let $\epsilon>0$.  Let $K_\epsilon \subset \MM_{S,\Sigma}(\kappa)$
be the set of translation surfaces
whose systole is at least $\epsilon$. Then there exists
$(k,\delta,h)$ such that for any translation surface
$\omega \in K_\epsilon$, the $(k,\delta,h)$-tunneling curves
for $\omega$ generate $H_1(S \backslash \Sigma; \Z)$.
\end{lemma}

%It may be proved that the constant
%$\delta = \frac{\sqrt{2}}{2} \epsilon$ works and is the best
%possible. As if $\zeta$ is a saddle connection of length
%$\epsilon$ in a translation surface $(S,\omega)$ which has
%the property that $\Re(\omega)(\zeta) = \Im(\omega)(\zeta)$,
%then the maximum width of a horizontal or vertical rectangle
%that cross $\zeta$ without hitting its endpoints has width
%$\frac{\sqrt{2}}{2} \epsilon$. It implies that
%$\frac{\sqrt{2}}{2} \epsilon$ is the best constant
%for \emph{every} surface with systole $\epsilon$.

\begin{proof}
  The set $A_{k,\delta,h}$ of surfaces in $\MM_{S,\Sigma}(\kappa)$
that admits a $(k,\delta',h')$-tunneling basis with
$\delta' > \delta$ and $h' < h$ is an open set. From compactness of
$K_\epsilon$ it is hence enough to prove that for any translation surface
in $\MM_{S,\Sigma}(\kappa)$, every closed curve in $S \setminus \Sigma$ is
homotopic to a $(k,\delta,h)$-tunneling curve,
for some $k$, $\delta$ and $h$.  Indeed, up to homotopy we may assume that
a closed curve is built by concatenating small (and hence embedded)
horizontal and vertical segments in alternation.  Those segments can then
be slightly thickened to build the desired cycle of rectangles.
\end{proof}

\subsection{Excursions in cusps}
In this section we prove
Theorem~\ref{thm:Veech_criterion_for_Veech_surfaces}.

We first give another formulation of
Theorem~\ref{thm:Veech_criterion_for_Veech_surfaces} in terms of
tunneling basis (in order
to use Theorem~\ref{thm:dual_Veech_criterion}). Let $\CC \subset
\MM_{S,\Sigma}(\kappa)$ be the $\SL(2,\R)$ orbit of a Veech surface,
let $\epsilon > 0$ be
small, and let $\CC_\epsilon$ be the set of surfaces in $\CC$
whose systole is at least $\epsilon$.
From Lemma~\ref{lem:nice_basis_in_compact_part}, we get $k$, $\delta$
and $h$ such that each translation structure
$\omega \in \CC_\epsilon$ has a $(k,\delta,h)$-tunneling basis in
$H_1(S \backslash \Sigma; \Z)$.
Using compactness and the finiteness of $(k,\delta,h)$-tunneling curves,
there exists a constant $M > 1$ such that for any translation
structure $\omega$ in $\CC_\epsilon$ any $(k,\delta,h)$-tunneling
basis $\{\zeta_j\}$
and any $(k,\delta,h)$-tunneling curve $\zeta$ for $\omega$, the
coefficients of $\zeta=\sum c_j \zeta_j$ with respect to the
basis satisfy $\sum |c_j|<M$.

Let $\omega \in \CC$ be a translation surface that admits
an eigenvalue $\nu$.
From Theorem~\ref{thm:dual_Veech_criterion}, there exists
$t_0$ such that for any $t$ larger than $t_0$, to each
$(k,\delta,h)$-tunneling curve $\zeta$ for $\omega_t$, we may associate
a unique $n_t(\zeta) \in \Z$ such that
$|\nu \Im(\omega)(\zeta) - n_t(\zeta)| < 1/(2M)$.
Let $t \geq t_0$ be such that $\omega_t \in \CC_\epsilon$.
If $\{\zeta_j\}$ is a $(k,\delta,h)$-tunneling basis and $\zeta=\sum c_j
\zeta_j$ is a $(k,\delta,h)$-tunneling curve, then
\[
|\nu \Im(\omega)(\zeta)-\sum c_j n_t(\zeta_j)|=|\sum c_j
(\nu \Im(\omega)(\zeta_j)-n_t(\zeta_j))|<\frac {1} {2M} \sum |c_j|<\frac {1}
{2},
\]
so that $n_t(\zeta)=\sum c_j n_t(\zeta_j)$.
As the $(k,\delta,h)$-tunneling curves form a basis in $\CC_\epsilon$,
the mapping $n_t: TC_{k,\delta,h}(\omega_t) \rightarrow \Z$ extends in a
unique way to a linear map $n_t:H_1(S \backslash \Sigma; \Z) \rightarrow \Z$.
The convergence to an integer element in
Theorem~\ref{thm:Veech_criterion_for_Veech_surfaces} is then equivalent
to the following statement.

\begin{lemma} \label{lem:Veech_criterion_for_Veech_surface}
  Let $\omega$ be a Veech surface in $\CC$ without vertical saddle
connections for which the translation
flow admits an eigenvalue $\nu$ and let $\omega_t=g_t \cdot \omega$.
Let $\epsilon>0$
and let $t_0$ and $n_t \in H^1(S \backslash \Sigma; \Z)$ be as above.
Then the family $(n_t)_{t \geq t_0, \omega_t \in \CC_\epsilon}$ is
eventually constant.
\end{lemma}

The proof of
Lemma~\ref{lem:Veech_criterion_for_Veech_surface} follows by
analyzing parts of Teichm\"uller geodesics that go off
$\CC_\epsilon$ because, by construction, $n_t$ is locally constant.

Until the end of this section,
fix $\epsilon>0$ such that the cusps of
$\CC$ are isolated in the complement of $\CC_\epsilon$ and that
the conclusion of
Corollary~\ref{cor:only_one_family_of_small_geodesics} holds (the former is
actually a consequence of the latter).
A \emph {cusp excursion}
of length $t>0$ is a segment of a Teichm\"uller orbit
$\omega_s=g_s \cdot \omega_0$, $s \in [0,t]$, such that $\omega_s \in
\CC_\epsilon$ only for $s=0,t$.  Note that a shortest saddle connection $\gamma$ at
the beginning of a cusp excursion is never horizontal or vertical,
and indeed we have $|\Im(\omega_0)(\gamma)|>|\Re(\omega_0)(\gamma)|>0$,
with the length of the cusp
excursion given by $t=\log \frac {|\Im(\omega_0)(\gamma)|}
{|\Re(\omega_0)(\gamma)|}$.

By our choice of $\epsilon$ (see Corollary~\ref{cor:only_one_family_of_small_geodesics})
, if $\omega$ belong to $\partial \CC_\epsilon=\{(S,\Sigma,\omega) \in \CC:\,
\sys(S,\Sigma,\omega)=\epsilon\}$, then $S$ admits a canonical decomposition
as a finite union of maximal cylinders $C_i$, $1 \leq i \leq c$, with waist
curve $\gamma_i$ parallel to the saddle connection $\gamma$.
%If $\gamma$ is not horizontal (for instance,
%if $\omega$ is the beginning of a cusp
%excursion), we will choose its sign so that
%$\Im(\omega)(\gamma)>0$, and similarly we will choose the signs of the
%$\gamma_i$ so that $\Im(\omega)(\gamma_i)>0$, $1 \leq i \leq c$.

\begin{lemma} \label{lem:cusp_excursion}
For any $(k,\delta,h)$ with $\delta>0$ small enough and $h>0$ big enough,
there exists an integer $k' \geq k$ with the following property.
Let us consider a cusp excursion $(\omega_s)_{0 \leq s \leq t}$ of length $t$. Let $\gamma$ be a saddle connection of shortest length on $\omega_0$.
Let $\kappa$ be the sign of $\frac {\Im(\omega_0)(\gamma)}
{\Re(\omega_0)(\gamma)}$ and let
$\displaystyle m_i = \left\lfloor \frac{e^t}{\mu(C_i)} \right\rfloor$,
where $\mu(C_i)$ is the modulus of the cylinder $C_i$ in the canonical
decomposition of $(S,\Sigma,\omega_0)$.  For each
$(k,\delta,h)$-tunneling curve $\zeta$ for $\omega_0$ the class
$\displaystyle  \zeta - \kappa \sum_{i=1}^c m_i \langle \zeta, \gamma_i
\rangle \gamma_i$ is $(k',\delta,h)$-tunneling for $\omega_t$ and for any
integers $\ell_i$ such that $0 \leq \ell_i \leq m_i$ the classes
$\displaystyle \zeta -
\kappa \sum_{i=1}^c \ell_i \langle \zeta, \gamma_i \rangle
\gamma_i$ are $(k',\delta)$-tunneling for $\omega_0$.
\end{lemma}

\begin{proof}%[Proof of Lemma~\ref{lem:cusp_excursion}]

%Let $\epsilon$ be as in the statement of the lemma. We first remark that
%the value of $(k,\delta,h)$ is not relevant because
%for any $(k_1,\delta_1,h_1)$ and $(\delta_2,h_2)$ there exists
%$k_2$ such that $TC_{(k_1,\delta_1,h_1)} \subset TC_{(k_2,\delta_2,h_2)}$.
%So we choose $\delta$ to be smaller than $\sqrt{2}/2 \epsilon$ and
%we assume that $h$ is larger than any width or height of cylinders
%that appear in the cylinder decomposition of a translation surface
%in $\CC$ with systole $\epsilon$. We will fix $k$ later.

%Let $(S,\Sigma,\omega_0)$ with systole $\epsilon$ and let
%$[\omega_0,\omega_t]$ be the associated cusp excursion
%out of $\CC_\epsilon$.

We first reduce our study to (not necessarily closed) paths inside a single
cylinder. Let $X = \{x_1,x_2,\ldots, x_p\}$ be the middle points of saddle
connections in the direction of the shortest saddle connection $\gamma$ for $\omega_0$.
We call \emph{transversal} a flat geodesic segment $\gamma'$
that joins two points of $X$ and is
disjoint from saddle connections parallel to $\gamma$.
Any curve
in $(S,\Sigma,\omega_0)$ is freely homotopic in $S \backslash \Sigma$ to a
concatenation of transversals $\gamma'_i$
such that $|\Im(\omega_0)(\gamma_i')| < 2R$ and
$|\Re(\omega_0)(\gamma_i')| < 2R$
where the constant $R$ may be choosen independently of
$(S,\Sigma,\omega)$ in $\partial \CC_\epsilon$.
Moreover, if $\zeta$ is a $(k,\delta,h)$-tunneling curve for $\omega_0$, the
minimal number of pieces is uniformly bounded in terms of $\epsilon$, $k$,
$\delta$ and $h$.

Let us fix $x$ and $y$ on the boundary of some cylinder $C_i$ and
denote by $T(x,y)$ the set of transversals
that join $x$ to $y$.
%Let $\gamma$ be the waist curve of the same cylinder.
%The set $C(x_i,x_j)$ is equal in the relative homology $H_1(S \setminus
%\Sigma, \{x_i,x_j\};
%\Z)$ to $\{\zeta + n \gamma; n \in \Z\}$ where $\zeta$ is any element in
%$C(x_i,x_j)$.
We build rectangles around the curves in $T(x,y)$ in order to be able
to reconstruct a cycle of rectangles.
A $(k',\delta,h)$-\emph{path of rectangles} for $x$ and $y$ is a set
of rectangles $H_1$, $V_1$, \ldots, $H_{k'}$, $V_{k'}$, $H_{k'+1}$
such that 
\begin{itemize}
\item the height of $H_j$ is $\delta$ and its width is $w_j \geq \delta$;
\item the width of $V_j$ is $\delta$
and its height $h_j$ satisfies $\delta \leq h_j \leq h$;
\item $H_j^*(\Re(\omega)) = \pm dx$ and $H_j^*(\Im(\omega)) = dy$;
\item $V_j^*(\Re(\omega)) = dx$ and $V_j^*(\Im(\omega)) = \pm dy$;
\item each rectangle $H_j$ is embedded in the surface;
\item for each $j$, $H_j ([0,\delta] \times [0,\delta]) = V_{j-1}([0,\delta] \times [h_j-\delta,h_j])$
	and
	$H_j ([w_j-\delta,w_j] \times [0,\delta]) = V_j([0,\delta] \times [0,\delta])$;
  \item $H_1(\delta/2,\delta/2) = x$ and $H_{k'+1}(w_{k'+1}-\delta/2, \delta/2) =
y$.
\end{itemize}
As we did for cycles of rectangles, to a $(k',\delta,h)$-path of rectangles
for $x$ and $y$ we may associate 
its homology class in $H_1(S \setminus \Sigma, \{x,y\}; \Z)$.

%Given a decomposition of a $(k,\delta,h)$-tunneling
%curve into a minimal number of transversals and
%a $(k',\delta,h)$-path of rectangles for each
%transversal we may build a $(k'',\delta,h)$-cycle of rectangles
%where $k''$ only depends on $k$, $k'$, and $\epsilon$.

Let us fix a transversal $\gamma'$ joining $x$ and $y$ inside some cylinder
$C_i$ and such that $|\Im(\omega_0)(\gamma')|<2R$ and
$|\Re(\omega_0)(\gamma')|<2R$.  Since any $(k,\delta,h)$-tunneling curve
can be decomposed into a uniformly bounded number of such transversals, it
will be enough for us to prove that there exists $k'>0$ (only depending on
$\epsilon,\delta,h,R$) with the following properties:
\begin{enumerate}
\item The transversal $\gamma'' \in T(x,y)$ in the class of
$\gamma'-\kappa m_i \langle \gamma',\gamma_i
\rangle \gamma_i$ is a $(k',\delta,h)$-tunneling path for $\omega_t$;
\item For every $0 \leq \ell \leq m_i$, the class of
$\ell \gamma_i$ is $(k',\delta)$-tunneling for $\omega_0$.
\end{enumerate}
Note that the width $w(C_i)$, the height $h(C_i)$, and the modulus,
$\mu(C_i)$ are all bounded away from zero and infinity, independently of
$i$, through $\partial \CC_\epsilon$.  In particular, we may assume that
$w(C_i)$ and $h(C_i)$ are bigger than $10 \delta$.  Note also that
$\langle \gamma',\gamma_i\rangle=\pm 1$.

Write $\frac
{\Im(\omega_0)(\gamma)} {\Re(\omega_0)(\gamma)}$ as $\kappa
\tan \theta$ with
$\frac {\pi} {4}<\theta<\frac {\pi} {2}$.
We may assume that the cusp excursion of the surface $(S,\Sigma,\omega_0)$ in
$\CC_\epsilon$ is long enough, so that in particular $m_i \geq 10$ and $\tan
\theta>\frac {2 R} {h(C_i)}$.

Let us first show the second property.  It is enough to show that for $0
\leq \ell \leq \frac {3 m_i} {4}$, $\ell \gamma_i$ can be represented by a
$(2,\delta)$-tunneling curve.  Let $\hat C_i$ denote the ``core of $C_i$'',
obtained by removing a $\frac {h(C_i)} {8}$-neighborhood of its boundary.
If $l \leq \frac {3 m_i} {4}$,
then we can represent $\ell \gamma_i$ by a
concatenation $\gamma'_i$ of a vertical path of length $\ell w(C_i) \sin \theta$
and a horizontal path of length $\ell w(C_i) \cos \theta$ inside
$\hat C_i$.\footnote {To
see that such a concatenation lies inside $\hat C_i$, note that the maximal
length of a vertical path in $\hat C_i$ is exactly $\frac {3} {4}
\frac {h(C_i)} {\cos \theta}$ and $\frac {3} {4}
m_i w(C_i) \sin \theta \leq \frac {3} {4} h(C_i) \frac {\sin^2 \theta}
{\cos \theta} \leq \frac {3} {4}
\frac {h(C_i)} {\cos \theta}$.}  It easily follows that the
$\delta$-enlargement of the horizontal part of $\gamma'_i$ is embedded in $C_i$,
while the $\delta$-enlargement of the vertical part of $\gamma'_i$ is contained
in $C_i$, so that $\gamma$ is $(2,\delta)$-tunneling.

Let us now show the first property.  By compactness considerations, it will
be enough to show that the
transversal $\gamma'' = \gamma' - \kappa m_i \langle \gamma',\gamma_i\rangle \gamma_i$ has bounded length with respect to $\omega_t$.  Up to
changing the orientation of $\gamma'$, we may assume that
$\langle \gamma',\gamma_i \rangle=1$.
Then, for the imaginary part we have
\[
|\Im(\omega_t)(\gamma'-\kappa m_i \gamma_i)|=e^{-t}
|\Im(\omega_0)(\gamma'-\kappa m_i
\gamma_i)|<e^{-t} 2 M+e^{-t} \left \lfloor \frac {e^t} {\mu(C_i)} \right
\rfloor \leq 2 M+\frac {1} {\mu(C_i)}.
\]
For the real part, note first that
\begin{align*}
\Re(\omega_t)(\gamma' - \kappa m_i \gamma_i) 
 &= e^t \Re(\omega_0)(\gamma' - \kappa m_i \gamma_i) \\
 &= \frac{|\Im(\omega_0)(\gamma_i)|}{|\Re(\omega_0)(\gamma_i)|}
\left(\Re(\omega_0)(\gamma')- \kappa m_i \Re(\omega_0)(\gamma_i)\right) \\
 &= \pm |\Im(\omega_0)(\gamma_i)|
\left(\frac{\Re(\omega_0)(\gamma')}{\Re(\omega_0)(\gamma_i)} - \kappa m_i \right).
\end{align*}
Recall that $m_i=\lfloor \frac
{h(C_i)} {u(C_i)} \tan \theta \rfloor$ while $\Im(\omega_0)(\gamma_i)$ is
uniformly bounded.
In order to conclude, let us show that $\frac {\Re(\omega_0)(\gamma')}
{\Re(\omega_0)(\gamma_i)}$ is at a uniformly bounded distance from
$\kappa \frac
{h(C_i)} {w(C_i)} \tan \theta$.  Using that $\langle \gamma',\gamma_i
\rangle=1$ and that $\tan \theta>\frac {2 M} {h(C_i)}$,
we see that $\Re(\omega_0)(\gamma')$ and
$\Im(\omega_0)(\gamma_i)$ have the same sign,\footnote {Indeed, let us
consider an horizontal path $\gamma''$
in $C_i$ joining the boundaries of $C_i$, which is homotopic to $\gamma'$
relative to $\partial C_i$.
Since $\Im(\omega_0)(\gamma')<2M$, the
condition $\tan \theta>\frac {2 M} {h(C_i)}$ implies that
the sign of $\Re(\omega_0)
(\gamma')$ is the same as the sign of $\Re(\omega_0)(\gamma'')$, and since
$\langle \gamma'',\gamma_i \rangle=\langle \gamma',\gamma_i \rangle=1$, this
must have the same sign as $\Im(\omega_0)(\gamma_i)$.}
which implies that $\frac {\Re(\omega_0)(\gamma')}
{\Re(\omega_0)(\gamma_i)}$ and $\kappa \frac
{h(C_i)} {w(C_i)} \tan \theta$ have the same sign as well.
We have $|\Re(\omega_0)(\gamma_i)|=w(C_i) \cos \theta$ and
\[
|\Re(\omega_0)(\gamma')| \pm |\Im(\omega_0)(\gamma')| \cot \theta=
\frac {h(C_i)} {\sin \theta},
\]
so that
\[
\frac {|\Re(\omega_0)(\gamma')|} {|\Re(\omega_0)(\gamma_i)|}=\frac {h(C_i)}
{w(C_i)} \frac {1} {\sin \theta \cos \theta} \mp \frac
{|\Im(\omega_0)(\gamma')|} {w(C_i) \sin \theta}.
\]
It follows that
\[
\frac {|\Re(\omega_0)(\gamma')|} {|\Re(\omega_0)(\gamma_i)|}-\frac {h(C_i)}
{w(C_i)} \tan \theta=\frac {h(C_i)}
{w(C_i)} \cot \theta \mp \frac
{|\Im(\omega_0)(\gamma')|} {w(C_i) \sin \theta}.
\]
Since $h(C_i)$ is uniformly bounded away from infinity, $w(C_i)$ is
uniformly bounded away from $0$, $\cot \theta<1$, $\sin \theta>2^{-1/2}$ and
$|\Im(\omega_0)(\gamma')|<2M$, the result follows.
\end{proof}

We now prove how Lemma~\ref{lem:cusp_excursion} may be used to
conclude the proof of Lemma~\ref{lem:Veech_criterion_for_Veech_surface}.
 \begin{proof}[Proof of Lemma~\ref{lem:Veech_criterion_for_Veech_surface}]
% Let $\epsilon$ be fixed as in Lemma~\ref{lem:cusp_excursion}
Let $(k,\delta,h)$ with $\delta$ small enough and $h$ large enough so
that every surface in $\CC_\epsilon$ admits a $(k,\delta,h)$-tunneling basis,
and such that for every surface in $\partial \CC_\epsilon$,
the waist curves $\gamma_i$ of the canonical cylinder decomposition
are $(k,\delta,h)$-tunneling.
Let $k' \geq k$ be such that the conclusion of
Lemma~\ref{lem:cusp_excursion} holds.
Let $M'$ be, as in the begining of the section, an upper bound for $\sum |c_\alpha|$ over all expressions $\sum
c_\alpha \zeta_\alpha$ of a $(k',\delta,h)$-tunneling curve in a
$(k',\delta,h)$-tunneling basis.
Let $(S,\Sigma,\omega)$ be a surface that admits an
eigenvalue $\nu$. We know from Theorem~\ref{thm:dual_Veech_criterion}
that there exists a time $t_0$ such that for every $t \geq t_0$ and every
$(k',\delta)$-tunneling curve $\alpha$ for $\omega_t=g_t \cdot \omega$
we have $\dist(\nu \Im(\omega)(\alpha),\Z) < 1/(4M')$.

Recall that for $t \geq t_0$ such that $\omega_t \in \CC_\epsilon$,
we may define $n_t \in H^1(S \backslash \Sigma; \Z)$ from the
nearest integer vectors of elements of $(k',\delta,h)$-tunneling
basis. By construction, $n_t$ remains constant in interval of times
for which $\omega_t \in \CC_\epsilon$.  Let $t \geq t_0$ be such that
$\omega_t$ is the beginning of a cusp excursion of length $\tau$.
We prove that $n_{t+\tau} = n_t$.

From Lemma~\ref{lem:cusp_excursion}, we know that there exist basis
$\{\zeta^0_j\}$ and $\{\zeta^1_j\}$ of $H_1(S \backslash \Sigma; \Z)$
such that the $\zeta^0_j$ are $(k,\delta,h)$-tunnneling for $\omega_t$,
the $\zeta^1_j$ are
$(k',\delta,h)$-tunneling for $\omega_{t+\tau}$ and for each $j$,
$\displaystyle \zeta^1_j - \zeta^0_j =
-\kappa \sum_{i=1}^c m_i \langle \zeta^{0}_j, \gamma_i \rangle \gamma_i$
is a sum of multiples of
waist curves of cylinders in the canonical decomposition of
$(S,\Sigma,\omega_t)$.  Moreover, each partial
sum $\displaystyle \zeta^\ell_j=\zeta^0_j-\kappa
\sum_{i=1}^c \ell_i \langle \zeta^{0}_j,
\gamma_i \rangle \gamma_i$, with $0 \leq \ell_i \leq m_i$,
is $(k',\delta)$-tunneling for
$\omega_t$.

Let us fix $j$. We know that $\zeta^0_j$ and $\zeta^1_j$ are $(k',\delta,h)$-$\omega_t$ tunneling. Hence
both $\nu \Im(\omega)(\zeta^0_j)$ and $\nu \Im(\omega)(\zeta^1_j)$ are $1/(4M')$ close from two integers
$v^0_j = n_t(\zeta^0_j)$ and $v^1_j = n_{t+\tau}(\zeta^1_j)$.
The curves $\zeta^\ell_j$ are $(k',\delta,\infty)$-tunneling for $\omega_t$ and
hence $\nu \Im(\omega)(\zeta^\ell_j)$ is also $1/(4M')$-close to an integer $w^\ell_j$.
Two consecutive interpolating curves  $\zeta^{\ell}_j$ and $\zeta^{\ell'}_j = \zeta^{\ell}_j \pm \gamma_i$
differ by $\pm \gamma_i$. The waist curves $\gamma_i$ are $(k',\delta,h)$-tunneling for
both $\omega_t$ and $\omega_{t+\tau}$ and hence $1/(4M')$
close to the integer $n_t(\gamma_i) = n_{t+1}(\gamma_i)$.
Hence we have $w^{\ell'}_j = w^\ell_j \pm n_t(\gamma_i)$ from which follows that $n_t(\zeta^1_j) = n_{t+\tau}(\zeta^1_j)$.

We have shown that $n_t$ and $n_{t+\tau}$ coincide on the basis $\{\zeta^1_j\}_j$. They
are hence equal in $H^1(S\backslash \Sigma; \Z)$.
\end{proof}

% 3 - Veech criterion for Veech surfaces

\section{Generic weak-mixing} \label{sec:generic_weak_mixing}

\subsection{Weak-mixing in almost every direction}
Let $(S,\Sigma,\omega)$ be a non-arithmetic Veech surface, and let $\CC$
be its $\SL(2,\R)$ orbit.

Let $k$ be the holonomy field of $(S,\Sigma,\omega)$.
We recall from Section~\ref{section:Galois_conjugate} that the
vector space $H^1(S;\R)$ admits $[k\mathop:\Q]$ invariant planes $V^\sigma$ associated to
the embeddings $\sigma: k \rightarrow \R$.
Let $\displaystyle W = \bigoplus_{\sigma} V^\sigma$
and $\pi_\sigma: W \rightarrow V^\sigma$ be the natural projection.
The real vector space
$W$ is actually defined over $\Q$ and we denote $W_\Z = W \cap H^1(S;\Z)$
the integer lattice in $W$. By definition, $\pi_\sigma$
restricted to $W_\Z$
is injective (because $W$ is the smallest subspace defined over $\Q$
that contains $V=V^{\id}$).

By Theorem~\ref{thm:Veech_criterion_for_Veech_surfaces}, if a surface $u \in \CC$ 
 has no vertical saddle connection and admits an
eigenvalue, then there exists a vector $v \in H^1(S \backslash \Sigma; \Z)$
such that $v-\nu \Im(\omega)$ belongs to the stable space of the Kontsevich-Zorich cocycle.
Notice that in this case, we
necessarily have $v \in W_\Z$. Indeed,
$\Im (\omega) \in W$, so the projection of $v-\nu \Im (\omega)$ on the
quotient $H^1(S \setminus \Sigma;\R)/W$ is contained in $H^1(S
\setminus \Sigma;\Z)/W_\Z$. If the projection of $v-\nu \Im \omega$
would be non-zero then the projection of $A_n(u) \cdot (v-\nu \Im (\omega))$
would be a non-zero integer vector as well, and hence far away from zero.

For $u \in \CC$ with no saddle connection, let $E^s(u)$ be the stable space
of the Kontsevich-Zorich cocycle restricted to $W$, i.e. the set of vectors
that are asymptotically shrinked by the Kontsevich-Zorich cocycle.

Notice that if $v=0$ and $v - \nu \Im (\omega) \in E^s(u)$ then $\nu=0$
(since $\Im (\omega)$ generates the strongest unstable subspace for the
Kontsevich-Zorich cocycle). Any measurable eigenfunction with eigenvalue
$0$ must be constant by ergodicity (which follows from the recurrence of $g_t
u$ in $\CC$).

Thus, the set of $u \in \CC$ with no saddle
connection and that admits an eigenvalue is contained in
\[
\EE=\bigcup_{v \in W_\Z \backslash \{0\}} \bigcap_{\sigma \neq \id}
\EE(\pi_\sigma(v)),
\]
where
\[
\EE(w)=\{u \in \CC:\, w \in E^s(u)\}.
\]
Then, the exceptionality of eigenvalues for non-arithmetic surfaces (i.e. the
ones for which we have at least one $\sigma \not= id$) follows from the following result.
\begin{theorem} \label{thm:atoms}
Let $(S,\Sigma,\omega)$ be a Veech surface and $\CC$ its $\SL(2,\R)$ orbit.
Let $k$ be its holonomy field and $V^\sigma$ the conjugates of the tautological
space $V^{id} = \R \Re(\omega) \oplus \R \Im(\omega)$.

Then, for any $\sigma:k \to \R$ and any non-zero vector $w \in V^\sigma$, the
set $\EE^\sigma(w) = \{u \in \CC:\ w \in (E^s(u) \cap V^\sigma)\}$ has measure zero
with respect to the Haar measure on $\CC$.

In other words, the harmonic measure of the Kontsevich-Zorich cocycle on
$V^\sigma$ has no atom.
\end{theorem}
Using that $W_\Z$ is only countable, the above theorem implies
that $\{u \in \CC:\ \text{$u$ has an eigenvalue}\}$ has zero measure.
But as shown in Section~\ref{subsec:reduction_to_Markov}, this is equivalent to the fact that
$\{\theta \in S^1:\ \text{$r_\theta u$ has an eigenvalue}\}$ has zero
Lebesgue-measure.

We will not prove Theorem~\ref{thm:atoms} here. It will be refined by our
bound on Hausdorff dimension. It can be proved using the fact the stable 
space of the Kontsevich-Zorich cocycles can be studied through random
walks (see e.g. the work Chaika-Eskin~\cite{ChaikaEskin} or Eskin-Matheus~\cite{EskinMatheus})
and the fact that the Veech group $\Gamma$ and its conjugates $\Gamma^\sigma$ are 
non-elementary.

\subsection{Bound on Hausdorff dimension}
Now, assuming the results of Section~\ref{sec:markov_model} and~\ref{section:anomalous}, we derive an upper bound on the Hausdorff dimension of $\EE(w)$.

Let $x = (S,\Sigma,\omega)$ be a Veech surface and $\CC$ its $\SL(2,\R)$-orbit.
We consider a small segment of unstable horocycle $\Delta$ through $x$ and build
a Poincar\'e section $Q$ for the Teichm\"uller flow that contains $\Delta$ as in
Section~\ref{subsec:Markov_for_CC}. The return map to $Q$ induces a map $T: \Delta \rightarrow \Delta$ and
an associated Kontsevich-Zorich cocycle $A: \Delta \rightarrow \Sp(H^1(S,\Sigma; \R))$.
%$\{p(u,0);\, u \in \Delta\}$.
We will show that in this horocycle segment, the set of
surfaces for which the the vertical flow is not weak mixing has Hausdorff
dimension $d<1$. By Lemma~\ref{redu}, it will follow that for any surface in $\CC$,
the set of directions for which the directional flow is not weak mixing has Hausdorff dimension
$d$ as well.

Actually the map $T$ is only a partial map defined on $\Delta^\infty \subset \Delta$ which
is the set of all $u \in \Delta$ which are in the domain of $T^n$ for all $n
\in \N$ (equivalently, the set of surfaces for which the forward Teichm\"uller
geodesic return infinitely often to the Poincar\'e section $Q$).
For any $u \in \Delta^\infty$,
we may define the stable space $E^{\sigma,s}(u)$ of the
Kontsevich-Zorich restricted to $V^\sigma$, i.e. the set of all $v \in V^\sigma$
such that $A_n(u) \cdot v \to 0$ as $n \to \infty$.
Let similarly as above
\[
\EE^\sigma(w) = \{u \in \Delta \backslash \Delta^\infty:\ w \in E^{\sigma,s}(u)\}.
\]
By Theorem~\ref{thm:nr} we have $\HD(\Delta \setminus \Delta^\infty)<1$.
Hence our main result (Theorem~\ref{thm:generic_weak_mixing}) will follow once we show that
for each $\sigma:k \to \R$ there exists $d_\sigma < 1$ so that $\HD(\EE^\sigma(w)) \leq d_\sigma$.

Recall that for any $\sigma$, the Lyapunov exponents of $\overline{r} \lambda^\sigma$ of $A|V^\sigma$ is positive 
(see Lemma~\ref{lem:separation_LE_Veech_surfaces}). Here $\overline{r}$ is the
normalization due to the time change induced by considering a first return map
to $Q$ instead of the geodesic flow (by definition, $\overline{r}$ is also the
inverse of the Lyapunov exponent of $A|V^{id}$ since $\lambda^{id} = 1$).

By Lemma~\ref{lem:expansion_vs_Lyapunov} we know that the expansion constant of $A|V^\sigma$
is given by the Lyapunov exponent. We can then apply Theorem~\ref{thm:LD_for_expansion} and 
Theorem~\ref{thm:HD} to conclude.

\subsection{On the group of eigenvalues in exceptional directions}
Using Theorem~\ref{thm:Veech_criterion_for_Veech_surfaces}, we will show
that the Kronecker
factor (the maximal measurable almost periodic factor)
of the translation flow of a Veech surface is always small.
For arithmetic Veech surfaces (square tiled surfaces), we will see
that, in any minimal direction,
this factor actually identifies with a maximal torus quotient of that surface.
For non-aritmetic one, we obtain that the dimension of the Kronecker factor
is at most the degree of the holonomy field.

Let $(S,\Sigma,\omega)$ be a Veech surface, $V = \R \Re(\omega) \oplus \R \Im(\omega) \subset H^1(S;\R)$
the tautological bundle and $k$ its trace field. For each embedding $\sigma:k \rightarrow \R$ we
note $V^\sigma$ the Galois conjugate of $V$. The subspace $W = \bigoplus V^\sigma \subset H^1(S;\R)$
is defined over $\Q$ and has dimension $2 [k\mathop:\Q]$.

The field $k$ acts by multiplication on $H^1(S;\R)$ preserving $H^1(S;\Q)$
as follows: for $\lambda \in k$ consider the endomorphism of $H^1(S;\R)$ that
acts by multiplication by $\lambda^\sigma$ in $V^\sigma$. In particular, the set $\OO_k$ of
elements $\lambda \in k$ that preserves $H^1(S;\Z)$ forms an order (a $\Z$-module
of rank $[k\mathop:\Q]$, stable under multiplication). This phenomenon is actually
much deeper as the action of $k$ preserves the complex structure on $H^1(S;\C)$ and the
Hodge decomposition $H^1(S;\C) = H^{1,0}(S) \oplus H^{0,1}(S)$ into holomorphic and
anti-holomorphic one forms: Veech surfaces belong to 
 so-called~\emph{real multiplication loci},
see~\cite{McMullen2003} and \cite{BouwMoeller2010}.

Now we turn to the case of arithmetic surfaces and describe their maximal tori. Let $(S,\Sigma,\omega)$ be
a square tiled surface. Let $x$ be a point in $S$ and $\Lambda$ the subgroup of $\C$ generated by
the integration of $\omega$ along closed loops. Then we have a well defined map
$f:S \rightarrow \C / \Lambda$ defined by $f(y) = \int_\gamma \omega \mod \Lambda$ where $\gamma$ is
any path that joins $x$ to $y$. The intersection $H^1(S;\Z) \cap V$
is naturally identified with $H^1(\C/\Lambda; \Z)$ through $f^*$ and we call $\C / \Lambda$ together with
the projection $f$ the \emph{maximal torus} of $S$.
Note that $f$ is not necessarily ramified over only one point.

\begin{theorem} \label{thm:nb_eig}
  Let $(S,\Sigma,\omega)$ be a Veech surface, $k$ its holonomy field.
  Then in each minimal direction
the group of eigenvalues is finitely generated.
  Moreover, 
  \begin{itemize}
	\item If $(S,\Sigma,\omega)$ is arithmetic ($k=\Q$) then, in each minimal direction, all eigenfunctions of the flow of $S$
are lifts from the maximal torus of $S$. In particular, there are exactly 2 rationally independent continuous eigenvalues.
	\item If $(S,\Sigma,\omega)$ is non-arithemtic ($k \not= \Q$) then,
the ratio of any two non-zero eigenvalues for the translation flow of $S$ belongs to $k$.  In
particular, in each minimal direction, there are at most $[k\mathop:\Q]$
rationally independent eigenvalues.
  \end{itemize}
\end{theorem}

\begin{proof}
  Let us assume that $\nu \in \R$ is an eigenvalue of the flow of $(S,\Sigma,\omega)$.

  Let $W = \bigoplus V^\sigma$ and $W_\Z = W \cap H^1(S;\Z)$.
  Let $E^s \subset W$ be the stable space of the Kontsevich-Zorich cocycle restricted to $W$ and
  denote $E^{s,\sigma} = E^s \cap V^\sigma$. Note that $E^{s,\sigma}$ has dimension at most 1.
  From Theorem~\ref{thm:Veech_criterion_for_Veech_surfaces}, if $\nu$ is an eigenvalue of the flow,
  there exists $v \in W_\Z$ such that $\nu \Im\omega - v \in E^s$.  The map
  $\nu \mapsto v$ provides a isomorphism between the group of eigenvalues
  and a subgroup of $W_\Z$, so the group of eigenvalues is finitely
  generated.

Decomposing $v= \sum v_\sigma$
  with respect to the direct sum $W = \bigoplus V^\sigma$ we get
  \begin{itemize}
	\item $\nu \Im\omega - v_{id} \in E^{s,id}$;
	\item for any $\sigma \not= id$, $v_\sigma \in E^{s,\sigma}$.
  \end{itemize}
  In particular, if the dimension of $E^s$ is not maximal then there is no eigenvalue.

  The action of $\OO_k$ preserves the set of lines in each $V^\sigma$ and hence preserves (globally)
  the stable space $E^s$. In particular, if $\nu \Im\omega - v \in E^s$ then for any $\lambda \in \OO_k$
  we have $\lambda\nu \Im\omega - \sum \sigma(\lambda) v_\sigma \in E^s$. So the set of potential eigenvalues
  \[
  \Theta = \{\mu \in \R:\ \exists v \in W_\Z, \quad \mu\Im\omega - v \in E^s\}
  \]
  is stable under multiplication by $\OO_k$.

  If $k=\Q$, then we saw that $W_\Z$ is naturally identified with the cohomology of the maximal torus of $S$. As
  all eigenvalues are contained in $\Theta$, they all come from the maximal torus.

  Now, the rank of $\Lambda$ is the rank of $H^1(S,\Z) \cap (\R \Im(\omega) \oplus E^s(\omega))$. We know that
  for non arithmetic surfaces this rank can not be maximal and is hence at most $[k\mathop:\Q]$. But, as
  $\Theta$ is stable under multiplication by $\OO_k$ its rank is a multiple of $[k\mathop:\Q]$ and hence is
  $0$ or $[k\mathop:\Q]$. As the ratio of any two eigenvalues is the ratio of two elements of $\Theta$ it
  belongs to $k$.
\end{proof}

% 4 - proof of the principal result

\section{Markov model} \label{sec:markov_model}
We introduce a Markov model to study the geodesic flow on the
$\SL(2,\R)$-orbits of Veech surfaces. It will be used in
Section~\ref{section:anomalous} to get a bound on Hausdorff dimension of
exceptional sets. Our large deviation results mimic the case of random
indpendent variables. The condition that ensures enough independence is
the so-called \emph{bounded distorsion} property.

\subsection{Locally constant cocycles} \label{subsec:locally_const_cocycles}
Let $\Delta$ be a measurable space, and let $\mu$ be a finite probability
(reference) measure on $\Delta$.  Let $\Delta^{(l)}$, $l \in
\Z$ be a partition $\mu$-mod $0$ of
$\Delta$ into sets of positive $\mu$-measure.
Let $T:\Delta \to \Delta$ be a
measurable map such that $T|\Delta^{(l)}:\Delta^{(l)} \to \Delta$ is
a bimeasurable map (i.e. a map that admits an almost everywhere
defined measurable inverse).

Let $\Omega$ be the space of all finite sequences of
integers.  The length of $\l \in \Omega$ will be denoted by $|\l|$.

For $\l=(l_1,...,l_n) \in \Omega$, we let $\Delta^\l$ be the set of
all $x \in \Delta$ such that $T^{j-1}(x) \in \Delta^{l_j}$ for $1 \leq j
\leq n$.  We say that $T$ has \emph{bounded distortion} if there exists $C_0>0$ such
that every $\Delta^\l$ has
positive $\mu$-measure and
$\mu^\l=\frac {1}
{\mu(\Delta^{\l})} T^{|\l|}_* (\mu|\Delta^\l)$ satisfies
$\frac {1} {C_0} \mu \leq \mu^\l \leq C_0 \mu$.
In particular, for every $n \geq 1$,
$\mu_n=\frac {1} {n} \sum_{k=0}^{n-1} T^k_* \mu$ satisfies $\frac {1} {C_0}
\mu \leq \mu_n \leq C_0 \mu$. Taking a weak limit, one sees that
there exists an invariant measure $\nu$ satisfying $\frac {1} {C_0} \mu \leq
\nu \leq C_0\mu$.  It is easy to see that this
invariant measure is ergodic provided the $\sigma$-algebra of
$\mu$-measurable sets is generated (mod 0) by 
$\{\Delta^\l: \l \in \Omega\}$.

The bounded distorsion condition is a weakened form of independence.
It says that the conditional measures with respect to some past events are
comparable to the initial measure up to multiplicative constants.
Indeed, it can be seen that if $C_0=1$ in the definition of bounded
distorsion then the measure is invariant and is a Bernoulli measure, i.e. it
does not depend on the past (or equivalently if $\l=(l_1,l_2,\ldots,l_n)$ then
$\mu (\Delta^\l) = \mu(\Delta^{(l_1)}) \mu(\Delta^{(l_2)}) \ldots \mu
(\Delta^{(l_n)})$).

Let $H$ be a finite dimensional (real or complex)
vector space, and let $\SL(H)$ denote the space of linear automorphisms of
$H$ with determinant $1$.
Given $T: \Delta \rightarrow \Delta$ as above, we can define a locally constant $\SL(H)$-cocycle
over $T$ by specifying a sequence $A^{(l)} \in \SL(H)$, $l \in \Z$: Take
$A(x)=A^{(l)}$ for $x \in \Delta^{(l)}$ and
$(T,A):(x,w) \mapsto (T(x),A(x))$.  Then the cocycle iterates are given by
$(T,A)^n=(T^n,A_n)$ where $A_n(x)=A(T^{n-1}(x)) \cdots A(x)$.  Notice that
if $\l=(l_1,...,l_n)$ then $A_n(x)=A^\l$ for $x \in \Delta^\l$ with
$A^\l=A^{(l_n)} \cdots A^{(l_1)}$.

For a matrix $A$, let $\|A\|_+ = \max(\|A\|, \|A^{-1}\|)$.
We say that $T$ is \emph{fast decaying} if
there exists
$C_1>0$, $\alpha_1>0$ such that 
\[
\sum_{l:\mu(\Delta^{(l)}) \leq
\epsilon} \mu(\Delta^{(l)}) \leq
C_1 \epsilon^{\alpha_1},\qquad \text{for $0<\epsilon<1$},
\]
and we say that
$A$ is \emph{fast decaying} if
there exists $C_2>0$, $\alpha_2>0$ such that 
\[
\sum_{l:\|A^{(l)}\|_+ \geq n}
\mu(\Delta^{(l)}) \leq C_2 n^{-\alpha_2}.
\]
%There exists
%$C_1>0$, $\alpha_1>0$ such that $\sum_{\mu(\Delta^{(l)}) \leq
%\epsilon} \mu(\Delta^{(l)}) \leq C_1 \epsilon^{\alpha_1}$, $0<\epsilon<1$.

The fast decaying property of the map $T$ says that most of the measure is supported on pieces $\Delta^{(l)}$ of large measures.
In other words, it is a full-shift on an infinite alphabet but close to a full-shift on a finite alphabet.

The fast decaying property of cocycles is a crucial hypothesis for large deviations. It is equivalent to the
integrability of $\exp(\delta \log \|A\|_+)$ for $\delta$ small enough. It
implies in particular that $A$ is a $\log$-integrable cocycle with
respect to the invariant measure $\nu$, i.e. $\int \ln \|A\|_+ d\nu<\infty$.

Note that these two conditions would be automaticall
satisfied if instead of the infinite alphabet $\Z$ we had a finite one.

In our applications, $\Delta$ will be a
simplex in $\P \R^p$, i.e., the image of $\P
\R^p_+$ by a projective transformation, $\mu$ is the Lebesgue
measure, and $T|\Delta^{(l)}$ is a
projective transformation for every $l
\in \Z$.  In fact, we will be mostly interested in
the case $p=2$, and we will
use freely the identification of $\P \R^2$ with $\overline \R=\R \cup
\{\infty\}$.

Note that if there exists a simplex $\Delta \Subset \Delta'$ such that
the projective extension of $(T|\Delta^{(l)})^{-1}$ maps
$\Delta'$ into itself for every $l \in \Z$, then the bounded distortion
property holds, see section 2 of \cite {AvilaForni2007}.
In this case we will say that $T$ is a {\it projective expanding
map}.

\subsection{The Markov model for the geodesic flow on $\SL(2,\R)/\Gamma$}
\label{subsec:Markov_for_CC}
Now we construct a map $T: \Delta \rightarrow \Delta$ which is a first return map of the geodesic flow.
The bounded distorsion and fast decaying properties will follow from hyperbolic properties of the geodesic
flow. The associated Kontsevich-Zorich cocycle $A: \Delta \rightarrow \Sp(H^1(S,\Sigma;\R)$ is shown to 
have the fast decaying property.

Let $\CC=\SL(2,\R)/\Gamma$ where $\Gamma$ has finite covolume.  Let us fix any
point $x \in \CC$.  Then we can find a
small smooth rectangle $Q$ through $x$, which is transverse to
the geodesic flow and provides us with a nice Poincar\'e section, in the
sense that the first return map to $Q$ under the geodesic flow has a
particularly simple structure. All these properties will
follow from the following commutation rules in $\SL(2,\R)$
\[
g_t h^-_s g_{t} = h^-_{e^{-2t} s}
\quad \text{and} \quad
g_t h^+_u g_{-t} = h^+_{e^{2t} u}.
\]

More precisely, let $p:\R^2 \to \CC$ be given by $p(u,s)=h^-_s(h^+_u(x))$.
Then for any $\epsilon>0$, we can find $u_-<0<u_+$ and $s_-<0<s_+$ with
$u_+-u_-<\epsilon$ and $s_+-s_-<\epsilon$ such that, letting
$\Delta=(u_-,u_+) \subset \R$ and $\widehat \Delta=\{(u,s) \in \Delta
\times \R;\, s_-<\frac {s} {1+s u}<s_+\}$, we can take $Q=p(\widehat
\Delta)$.  It is clear that $Q$ is transverse to the geodesic flow. Let $F$
denote the first return map to $Q$.  Then
\begin{enumerate}
\item There exist countably many disjoint open
intervals $\Delta^{(l)} \subset \Delta$, such that the domain of $F$ is the
union of the $p(\widehat \Delta^{(l)})$, where
$\widehat \Delta^{(l)}=\widehat \Delta \cap (\Delta^{(l)} \times \R)$.
\item There exists a function $r:\bigcup \Delta^{(l)} \to \R_+$
such that if $(u,s)$ belongs to some $\widehat \Delta^{(l)}$ then the return
time of $p(u,s)$ to $Q$ is $r(u)$.  Moreover, $r$ is globally bounded away
from zero, and its restriction to each
$\Delta^{(l)}$ is given by the logarithm of the restriction of a
projective map $\overline \R \to \overline \R$.
\item 
There exist functions $T:\bigcup \Delta^{(l)} \to \Delta$ and
$S:\bigcup \Delta^{(l)} \to \R$ such that if $(u,s)$ belongs to some $\widehat
\Delta^{(l)}$ then $F(p(u,s))=p(T(u),S(u)-e^{-2r(u)}s)$.  Moreover, the
restriction of $T$ to each $\Delta^{(l)}$ coincides with the restriction
of a projective map $T_l:\overline \R \to \overline \R$, and the restriction
of $S$ to each $\Delta^{(l)}$ coincides with the restriction of an affine
map $S_l:\R \to \R$.
\item There exists a bounded open interval $\Delta'$ containing $\overline
\Delta$ such that $T_l^{-1}(\Delta') \subset \Delta'$ for every $l \in \Z$.
\end{enumerate}
The basic idea of the construction is to guarantee that the forward orbit of
the ``unstable''
frame $\delta_u(Q)=p \{(u,s) \in \partial \widehat \Delta;\, u=u_\pm\}$
and the backward orbit of the ``stable'' frame
$\delta_s(Q)=p\{(u,s) \in \partial \widehat \Delta;\, \frac {s} {1+s
u}=s_\pm\}$ never come back to $Q$.\footnote
{This is easy enough to do when $x$ is
not a periodic orbit of small period. 
From a segment of the unstable horcycle through $x$,
remove all $g_t$ pullback of the $C \epsilon$-neigorhood
of $x$. Because of hyperbolicity, the remainder is a small
Cantor set. Then $u_{\pm}$ are chosen so that $h^+_{u_\pm}(x)$ lies
in the complement of this set.
} 
This easily yields
the Markovian structure and the remaining properties follow from direct
computation (or, for the last property, by shrinking $\epsilon$).

\begin{rem} \label{rem:formulasSTr}
Given $u_0 \in \Delta^{(l)}$, knowledge of $T(u_0)$, $S(u_0)$ and
$r(u_0)$ allows one to easily compute $T$, $S$, and $r$ restricted to
$\Delta^{(l)}$.  Indeed, for $u \in \Delta^{(l)}$,
$g_{r(u_0)}(p(u,0))=h^+_{e^{2 r(u_0)} (u-u_0)} F(p(u_0,0))$.
To move it to $Q$, we must apply $g_{-t}$ where $t$ is bounded (indeed at
most of order $\epsilon$).  Using that $F(p(u_0,0))=p(T(u_0),S(u_0))$ one
gets $e^t=1+e^{2 r(u_0)} (u-u_0) S(u_0)$ and then the formulas
$$
e^{r(u)}=\frac {e^{r(u_0)}} {1+e^{2 r(u_0)} (u-u_0) S(u_0)},
$$
$$
T(u)=T(u_0)+\frac {e^{2 r(u_0)} (u-u_0)} {1+e^{2 r(u_0)} (u-u_0) S(u_0)}
$$
$$
S(u)=S(u_0) (1+e^{2 r(u_0)} (u-u_0) S(u_0)).
$$
Note that
\begin{equation} \label {der}
DT(u)=e^{2 r(u)},
\end{equation}
and that for every $l \in \Z$,
$r \circ T_l^{-1}:\Delta \to \R$ has uniformly bounded derivative.
\end{rem}

Note that $T$ is a projective expanding map with bounded distortion, so it
admits an ergodic invariant measure $\nu$ equivalent to Lebesgue measure.
In order to obtain an
upper bound on the Hausdorff dimension of the set of non-weak mixing
directions, we will also need to use that $T$ is fast decaying.  Using (\ref
{der}), we see that fast decay is implied by
the following well known exponential tail estimate on return times:
there exists $\delta>0$ (depending on $Q$) such that
$$
\int_{u_-}^{u_+} e^{-\delta r(u)} du<\infty.
$$

\begin{rem}
The exponential tail estimate is usually proved using a finite
Markov model for the full geodesic flow (as opposed to the infinite
Markov model for a Poincar\'e return map that we consider here).
However, it can also be proved using some more general information about the
geodesic flow.  

Details of this approach are carried out in \cite {AD2}, where it is used
to obtain exponential tails for the return time for a Markov
model of an arbitrary affine $\SL(2,\R)$-invariant measure
in moduli space.
\end{rem}

\begin{rem}
As remarked before, $r \circ T_l^{-1}:\Delta \to \R$ has uniformly bounded
derivative. 
This estimate can be iterated as follows. 
The equality $DT(u) = e^{2r(u)}$ does not depend on the fact that
it was a first return.  Let $r_n(u)=\sum_{k=0}^{n-1} r(T^k(u))$
be the $n$-th return time to $Q$.
Then we have $D T^n(u) = e^{2 r_n(u)}$ and formulas similar to the one
in Remark~\ref{rem:formulasSTr} for $T^n$ and $S^n$ hold.
Writing $T_\l = T_{l_n} \circ \cdots T_{l_1}: \Delta^\l \to \Delta$ it implies
that $r_n \circ T_\l^{-1}: \Delta \to \R$ has uniformly bounded derivative
independently of $n$.
\end{rem}

Let now $\CC=\SL(2,\R)/\Gamma$
be the $\SL(2,\R)$-orbit of some Veech surface.
Then the Kontsevich-Zorich cocycle over
$\CC$ gives rise to a locally constant cocycle over $T$ as follows.

Recall from Section~\ref{subsec:KZ} that there is a natural trivialization
of the Hodge bundle over $Q$ that preserve the integer lattice, the tautological
bundle $V=\R \Re\omega \oplus \R \Im\omega$ and its conjugates $V^\sigma$.
The induced Kontsevich-Zorich cocycle $A: \Delta \rightarrow \Sp(H^1(S;\R))$
is such that $A(y)$ depends continuously on $y \in p(\widehat \Delta_j)$. Because
it belongs to a discrete group it must be in fact constant. We let $A^{(l)}$
denote its value on $\Delta^{(l)}$.

If $\l=(l_1,...,l_n)$ with $n \geq 1$ then $F^n$ has a unique
fixed point $(S,\Sigma,\omega)=p(u_\l,s_\l)$ with
$(u_\l,s_\l)
\in \widehat \Delta^\l=(\Delta^\l \times \R) \cap \widehat \Delta$,
and $A^\l|V_x$ is hyperbolic with unstable direction $\Im(\omega)$,
stable direction $\Re(\omega)$, and Lyapunov exponent $r_n(u_\l)$. 
In particular, $\|A^\l|V\|$ is of order $e^{r_n(u_\l)}$ (up to uniformly
bounded multiplicative constants), since the angle
between $\Re(\omega)$ and $\Im(\omega)$ is uniformly bounded over $Q$.

Note that $\|A^\l\|$ is also of order $e^{r_n(u_\l)}$
(this follows for instance from Remark \ref {bound}).  In particular
($n=1$), the exponential tail estimate implies
that $A$ is fast decaying.

Since the geodesic flow on $\CC$ is ergodic, the
Lyapunov exponents of the Kontsevich-Zorich cocycle on $\CC$ (with
respect to the Haar measure) are the same as the Lyapunov exponents of the
locally constant cocycle $(T,A)$, with respect to the invariant measure
$\nu$, up to the normalization factor $\overline r=\int r(u) d\nu(u)$.

\subsection{Reduction to the Markov model} \label{subsec:reduction_to_Markov}
An eigenfunction $f:S \rightarrow \C$ with eigenvalue $\nu \in \R$ of a translation flow
$\phi_t:S \rightarrow S$ is a measurable function such that $f \circ \phi_t = e^{2\pi i \nu t}f$.
Note that if $f$ is a measurable or continuous eigenfunction for
the vertical flow on a translation surface $z=(S,\Sigma,\omega) \in
\MM_{S,\Sigma}(\kappa)$, then
$f$ is also an eigenfunction for $g \cdot z$
for any $g \in \SL(2,\R)$ which fixes the vertical direction, and in
particular for any $g$ of the form $h^-_s g_t$, $s,t \in \R$.

\begin{lemma} \label {redu}
Let $\CC$ be a closed $\SL(2,\R)$ orbit in some $\MM_{S,\Sigma}(\kappa)$.
Let $I$ be a non-empty open subset of
an $\SO(2,\R)$ orbit and let $J$ be a non-empty open subset of an unstable
horocycle.  Then:
\begin{enumerate}
\item For every $z_0 \in I$, there exists a diffeomorphism taking $z$ to $x$ from an open neighborhood $I' \subset I$ of $z_0$ to a subinterval of $J$, such that
the stable horocycle through $z$ intersects the geodesic through $x$,
\item For every $x_0 \in I$, there exists a diffeomorphism taking $z$ to $x$ from an open neighborhood $J' \subset J$ of $x_0$ to a subinterval of $I$, such that
the stable horocycle through $z$ intersects the geodesic through $x$.
\end{enumerate}
In particular, if $\Lambda$ is any subset of $\CC$ which is invariant by the
stable horocycle and geodesic flows (such as the set of translation surfaces
for which the vertical flow admits a continuous eigenfunction, or a
measurable but discontinuous eigenfunction),
$\HD(I \cap \Lambda)=\HD(J \cap \Lambda)$.

\end{lemma}

\begin{proof}
We prove the first statement.  Fix $z \in I$ and
some compact segment $J_0 \subset J$.  Then we
can choose $t_0$ large such that there exists $y \in J_0$ with
$g_{t_0}(y)$ close to $z$
(indeed as $t \to \infty$, $g_t \cdot I_0$
is becoming dense in $\CC$).  Thus for every $\theta \in \R$ close to
$0$ we can write $r_\theta z=h^-_s h^+_u
g_{t_0+t} y$ in a unique way with $s,t,u$ small, and moreover $\theta
\mapsto u$ is a diffeomorphism.

The second statement is analogous.
\end{proof}

% 5. Markov model

\section{Anomalous Lyapunov behavior, large deviations and Hausdorff dimension}
\label{section:anomalous}

We have seen so far that weak mixing can be established by ruling out
non-trivial
intersections of $\Im(\omega)$ with integer translates of the strong stable
space. This criterion can be rephrased in terms of
certain fixed vectors (projections of integer points on
Galois conjugates of the tautological bundle) lying in the strong stable
space. In particular, its iterate must see a non-positive
rate of expansion, instead of the expected rate (given by one of
the positive Lyapunov exponents).

In this section we introduce techniques to bound anomalous Oseledets
behavior in the setting of
locally constant cocycles with bounded distortion.
The Oseledets Theorem states that for a typical orbit, any vector will
expand precisely at the rate of some Lyapunov exponent.  For a given
vector, one can consider the minimum expansion rate
which can be seen with positive probability.  We will first show a (finite
time) upper bound on the probability of seeing less than such minimum
expansion.  Then we will show that such an estimate can be converted
into an upper
bound on the Hausdorff dimension of orbits exhibiting exceptionally small
expansion.

\subsection{Large deviations} \label{subsec:ld}
Let $(T,A)$ be a locally constant $\log$-integrable cocycle over a map $T:\Delta \to \Delta$ preserving a measure $\mu$.
The \emph{expansion constant} of $(T,A)$ is the maximal $c \in \R$ such that
for all $v \in \R^d \setminus \{0\}$ and for
$\mu$-almost every $x \in \Delta$ we have
\[
\lim_{n \to \infty} \frac {1} {n} \ln \|A_n(x) \cdot v\| \geq c.
\]
(The limit exists by Oseledets Theorem applied to $\mu$.)

\begin{theorem} \label{thm:LD_for_expansion}
Let $T: \Delta \rightarrow \Delta$ be a countable shift endowed with a measure $\mu$ with bounded distorsion and which is fast decaying.
Let $A: \Delta \rightarrow \SL(d,\R)$ be a fast decaying cocycle.
Then for every $c'$ smaller than the expansion constant of $A$,
there exist $C_3>0$, $\alpha_3>0$ such that for every unit vector $v \in
\R^d$,
\[
\mu \{u \in \Delta:\ \|A_n(x) \cdot v\| \leq e^{c' n}\} \leq C_3 e^{-\alpha_3 n}.
\]
\end{theorem}

Actually, the expansion constant is intimately related to Lyapunov exponents. Recall
that the cocycle $A: \Delta \rightarrow \SL(d,\R)$ is \emph{irreducible} if there is no
non-trivial subspaces of $\R^d$ that is invariant under the group 
generated by the matrices $(A^{(l)})_{l \in \Z}$.
\begin{lemma} \label{lem:expansion_vs_Lyapunov}
Let $T: \Delta \rightarrow \Delta$ be a countable shift with bounded distorsion and let
$A: \Delta \rightarrow \SL(d,\R)$ be an irreducible and $\log$-integrable cocycle.
Then the expansion constant of $A$ is its maximal Lyapunov exponent.
\end{lemma}

Let us first prove the lemma.
\begin{proof}
Let $\nu$ be the invariant measure of $T$ and $\lambda$ the
maximal Lyapunov exponent of $A$, i.e. the a.e. limit
\[
\lambda = \lim_{n \to \infty} \frac{\ln \|A_n(u)\|}{n}
\]
Let
\[
E_2(u) = \left\{w \in \R^d:\ \lim_{n \to \infty} \frac{\ln \|A_n(u) \cdot w\|}{n} < \lambda\right\}.
\]
By the Oseledets Theorem, for any  $w \in \R^d \setminus \{0\}$ for $\nu$-almost every $u \in \Delta$ we have that
\[
\lim_{n \to \infty} \frac{\ln \|A_n(u) \cdot w\|}{n} = \lambda
\]
unless $w$ belongs to $E_2(u)$.

Assume that there exists a vector $w \in \R^d \backslash \{0\}$ and a
subset of $\Delta$ of positive $\mu$-measure such that $w \in E_2(u)$.
Let $H \subset \R^d$, be a subspace of maximal dimension so that $H \subset E_2(u)$ for a positive
measure set of $u$ and let $u_0 \in \Delta$ be a density point in that set. Then applying $T^n$ for $n$ large enough
and using bounded distorsion, one can see that the space $(A_n(u_0) H) \subset E_2(u)$ for a subset
of $\Delta$ of measure arbitrarily close to $1$. By compacity of the grassmanian of $\R^d$,
one can chose $H'$ of the same dimension so that $H' \subset E_2(u)$ for a.e. $u \in \Delta$.

Now consider a finite word $\l$ and the subspace $H'' = (A^\l)^{-1} H'$. Then, by construction
for a.e. $u \in \Delta^{\l}$ the set $H''$ is contained in $E_2(u)$. If $H'' \not = H'$ then the subspace
$H'' \oplus H'$ would be contained in $E_2(u)$ on a positive measure set and of larger dimension. This
contradicts our initial choice of $H'$. Hence the space $H'$ is invariant under
all matrices $A^{\l}$ which contradicts the irreducibility assumption since
$0 < \dim H' \leq \dim E_2(u) < \dim H$.
\end{proof}

\begin{proof}[Proof of Theorem~\ref{thm:LD_for_expansion}]
Let $c$ be the expansion constant of $A$ and let $c' < c$.
For $v \in \R^d \setminus \{0\}$, let $I(x,n,v)=\frac {1} {n}
\ln_+ \frac {\|A_n(x) \cdot v\|} {\|v\|}$ and $I(\l,v)=\frac {1} {|\l|} \frac
{\ln_+\|A^\l \cdot v\|} {\|v\|}$ where $\ln_+(t) = \max(0,\ln(t))$.

Recall that because $T: \Delta \to \Delta$ has bounded distorsion, the invariant measures $\mu$ comes
with a family of measures $\mu^\l$ for $\l \in \Omega$ defined by
\[
\mu \left(\Delta_{\l\ \l'}\right) = \mu \left(\Delta_{\l}\right)\ \mu^{\l}\left(\Delta_{\l'}\right).
\]
They satisfy $1/C \mu \leq \mu^\l \leq C \mu$ for some constant $C$ uniform in $\l$ (see section~\ref{subsec:locally_const_cocycles}).

\medskip

The proof basically proceeds in two steps. We first prove that there exists an
integer $n_0$ for which we have some decay (see
equation~\eqref{eq:ld2} which is uniform in $v \in \R^d \backslash \{0\}$ and $\l \in \Omega$). Then we use the bounded distorsion property to
extend it to all $n$.

\medskip

We first claim that
\begin{equation} \label{eq:ld1}
\lim_{n \to \infty}
\sup_{v \in \R^d \setminus \{0\}}
\sup_{\l \in \Omega}
\mu^\l\{x \in \Delta:\ I(x,n,v)<c'\}=0.
\end{equation}
Indeed, the definition of the expansion constant $c$ gives that for any $c'' < c$ we have
\[
\sup_{v \in \R^d \setminus \{0\}} \lim_{n \to \infty}
\mu\{x:\ I(x,n,v)<c''\}=0,
\]
And we exchange quantifiers as follows (the argument comes from Lemma 3.3 of~\cite{AvilaForni2007}). Let
$C(x,v) = \inf_{n \geq 0} \frac{\|A_n(x) \cdot v\|}{e^{n c''} \|v\|}$.
By hypothesis, $C(x,v) > 0$ for all $v \in \R^d \backslash \{0\}$ and almost
every $x$. Moreover if $v_k \in \R^d \backslash \{0\}$ is a sequence
converging to $v$ then $C(x,v_k) \to C(x,v)$ for almost every $x$.

Let $\alpha$ be a positive real constant.
Using the Fatou lemma, we have that
\begin{align*}
\limsup_{v_k \to v} \mu \{x: C(x,v_k) < \alpha\}
& = \limsup_{v_k \to v} \int_\Delta 1_{C(x,v_k) < \alpha} d\mu(x) \\
& \leq \int_\Delta \left(\limsup_{v_k \to v} 1_{C(x,v_k) < \alpha}\right) d\mu(x)
= \int_\Delta 1_{C(x,v)} d\mu(x).
\end{align*}
That is to say, the function $v \mapsto \mu \{x: C(x,v) < \alpha\}$ is upper semi-continuous.
Now for every $v$ there exists $\alpha_\epsilon(v)$ so that
$\mu \{x: C(x,v) \leq \alpha_\epsilon(v) \} < \epsilon$.
By upper semi-continuity, this is also true in a neighborhood of $v$ and we can conclude
using compactness (and the fact that $c''$ was arbitrary).

The fact that the convergence is uniform in $\l \in \Omega$ in
equation~\eqref{eq:ld1} follows directly from bounded distorsion.
This conclude the proof of the claim.

\medskip

Because $\int_\Delta I(x,n,v) d\mu^\l(x) \geq c'' \mu^\l \{x: I(x,m,v) \geq c''\}$ (Markov inequality), we can
find $n_0$ and $\kappa'$ so that
\begin{equation} \label{eq:ld2}
\sup_{v \in \R^d \setminus \{0\}} \sup_{\l \in \Omega} \int c'-I(x,n_0,v) d\mu^\l
<-\kappa'.
\end{equation}
Now, we transform this estimate into an exponential form.
By fast decay of the cocycle $A$, there exists $C'>0$, $\delta'>0$ such that
for all $|s|<\delta'$, $1 \leq n \leq n_0$ we have
\[
\sup_{v,\l} \int_\Delta |e^{s n (c'-I(x,n,v))}| d\mu^\l(x) \leq C'.
\]
Hence $\phi_{n,\l,v}: s \mapsto \int e^{s n (c'-I(x,n,v))} d\mu^\l(x)$
are uniformly bounded holomorphic functions of $|s|<\delta'$ for $n \leq n_0$, $\l \in \Omega$ and $v \in \R^d \backslash \{0\}$. Note that
$\phi_{n,\l,v}(0) = 1$ and $\phi'_{n,\l,v}(0) = n (\int c' - I(x,n,v) d\mu^\l(x))$.
Thus, using~\eqref{eq:ld2} that gives an upper bound on the derivative $\phi_{n_0,\l,v}(0)$,
there exists $\delta$, $\kappa$ (with $0 < \delta \leq \delta'$ and $0 < \kappa \leq \delta \kappa'$) so that
\begin{equation} \label{eq:proof_n0-1}
\sup_{v, \l} \int
e^{\delta n_0 (c'-I(x,n_0,v))} d\mu^\l<e^{-\kappa n_0},
\end{equation}
while, for every $1 \leq n \leq n_0-1$,
\begin{equation} \label{eq:proof_n0-2}
\sup_{v, \l} \int e^{\delta n (c'-I(x,n,v))} d\mu^\l<2 e^{-\kappa n}.
\end{equation}

\medskip

We have proved so far that at the scale $n_0$, we have an exponential decay. We now
use the bounded distorsion property to extend it for all times.
Note that from submultiplicativity of the norm, we have that $(m+n) I(x, m+n, v) \leq n I(x,n,v) + m I(T^n x, m, A_n(x) \cdot v)$.
Hence, for all positive integers $m,n$ we have
\begin{align} \label{eq:proof_prod}
\int e^{\delta (n+m) (c'-I(x,n+m,v))} d\mu(x)&\leq\sum_{|\l|=n}
\mu(\Delta^\l) e^{\delta n (c'-I(\l,v))}
\int e^{\delta m (c'-I(x,m, A^\l \cdot v))} d\mu^\l(x)\\
\nonumber
&\leq
\int e^{\delta n (c'-I(x,n,v))} d\mu(x) \sup_{|\l|=m} \int e^{\delta
m(c'-I(x,m,A^\l \cdot v))} d\mu^\l.
\end{align}
Now given any $n$, we do the euclidean division $n = k n_0 + r$.
Using the product formula~\eqref{eq:proof_prod} several times together with the
estimates~\eqref{eq:proof_n0-1} (for the parts in $n_0$)
and~\eqref{eq:proof_n0-2} (for the rest $r$) we obtain that
\[
\mu \{x: I(x,n,v) \leq c'\} \leq \int e^{\delta n (c'-I(x,n,v))} d\mu(x) \leq 2 e^{-\kappa n}.
\]
\end{proof}

\begin{rem}

The previous theorem can be somewhat refined: If $A$ is fast decaying and
for some vector $v \in \R^d
\setminus \{0\}$ we have $\lim \frac {1} {n} \ln \|A_n(x) \cdot v\|>
c'$ for a positive $\mu$-measure set of $x \in \Delta$, then the
$\mu$-measure of the set of $x$ such that $\frac {1} {n} \ln \|A_n(x) \cdot
v\| \leq c'$ is exponentially small in $n$.  This can be proved by reduction
to the setting above after taking the quotient by an appropriate invariant
subspace.

\end{rem}

\subsection{Hausdorff dimension}
The next result shows how to convert Theorem \ref{thm:LD_for_expansion}
into an estimate on Hausdorff dimension.  We will assume that $T:\Delta \to
\Delta$ is a transformation with bounded distortion, $\Delta$ is a
simplex in $\P \R^p$ for some $p \geq 2$ and $T|\Delta^{(l)}$ is a
projective transformation for every $l \in \Z$.

Recall that $T$ is fast decaying if there exists constants $\alpha_1 > 0$ and $C_1 > 0$ so that
for any $\epsilon > 0$
\[
\sum_{\mu(\Delta^{(l)}) \leq \epsilon} \mu(\Delta^{(l)}) \leq C_1 \epsilon^\alpha_1.
\]
\begin{theorem} \label{thm:HD}
Assume that $T$ is fast decaying.
For $n \geq 1$, let $X_n \subset \Delta$ be
a union of $\Delta^\l$ with $|\l|=n$,
and let $X=\liminf X_n$.  Let
\[
\delta = \limsup_{n \to \infty} \frac {-\ln \mu(X_n)} {n}.
\]
Then $\HD(X) \leq p-1-\min(\delta,\alpha_1)$ where $\alpha_1$
is the fast decay constant of $T$.
\end{theorem}

We will need a preliminary result:

\begin{lemma}
Assume that $T$ has bounded distortion and is fast decaying.  Then for
$0<\alpha_4<\alpha_1$, there
exists $C_4>0$ such that for every $n \geq 1$, we have
\[
\sum_{|\l|=n} \mu(\Delta^\l)^{1-\alpha_4} \leq C_4^n.
\]
\end{lemma}

\begin{proof}
The fast decay of $T$ implies that for any $k$,
\[
\sum_{2^{-k-1} \leq \mu(\Delta^{(l)}) \leq 2^{-k}} \mu(\Delta^{(l)})
\leq C_1 2^{-\alpha_1 k}.
\]
Hence
\[
\sum_{2^{-k-1} \leq \mu(\Delta^{(l)})^{1-\epsilon} \leq 2^{-k}} \mu(\Delta^{(l)})
\leq 2^{-(1-\epsilon)k} \frac{C_1 2^{-\alpha_1 k}}{2^{-k-1}}.
\]
Summing over $k$ we get
\[
\sum_{l \in \Z} \mu(\Delta^{(l)})^{1-\epsilon}
\leq 2 C_1 \sum_{k \geq 0} 2^{(\epsilon-\alpha_1) k}.
\]
It follows that for every $\l$,
\[
\sum_{l \in \Z} \mu^\l(\Delta^{(l)})^{1-\epsilon} \leq C_4.
\]
On the other hand, it is clear that
\[
\sum_{|\l'|=n+1} \mu(\Delta^{\l'})^{1-\epsilon}=\sum_{|\l|=n}
\mu(\Delta^\l)^{1-\epsilon} \sum_{l \in \Z} \mu^\l(\Delta^{(l)})^{1-\epsilon}.
\]
The result follows by induction.
\end{proof}

\begin{proof}[Proof of Theorem~\ref{thm:HD}]
Notice that there exists $C'>0$ such that if $0<\rho \leq \rho'$, then any
simplex with Lebesgue measure $\rho'$ is contained in the union of $C'
\frac {\rho'} {\rho^{p-1}}$ balls of diameter $\rho$.

Fix $0 < \delta' < \min(\delta,\alpha_1)$. By definition, there are
infinitely many $n$ so that $\mu(X_n) < e^{- \delta' C n}$. Fix such $n$
and fix $\delta'<\alpha_4<\alpha_1$. Let $C_4>0$ be as in the previous lemma,
and let
$C>0$ be such that $C_4 e^{- C (\alpha_4 - \delta')}<1$.
We are going to find a cover $\{B_i\}$ of $X_n$ by balls of
diameter at most $e^{-C n}$ satisfying
\[
\sum_i \diam(B_i)^{p-1-\delta'} \leq 2 C',
\]
showing that $\HD(\liminf X_n) \leq p-1-\delta'$.

Let $X_n=Y_n \cup Z_n$, where $Y_n$ is the union of those
$\Delta^\l$ with $|\l|=n$ such that
$\mu(\Delta^\l)>e^{-C n}$ and $Z_n$ is the complement.
It follows that $Y_n$ can be covered with at most
$C' \mu(Y_n) e^{(p-1) C n}$ balls of diameter $e^{-C n}$.  This cover
$\{B^Y_i\}$ satisfies
\[
\sum_i \diam(B^Y_i)^{p-1-\delta'} \leq C' \mu(X_n) e^{\delta' C n} \leq C'.
\]
  
Let us cover each $\Delta^\l \subset Z_n$ by the smallest possible number of
balls of diameter $\mu(\Delta^\l)$.  The resulting cover $\{B^Z_i\}$
of $Z_n$ then satisfies
\begin{align}
\sum_i \diam(B^Z_i)^{p-1-\delta'} &\leq \sum_{|\l|=n,\mu(\Delta^\l) \leq e^{-C
n}}
C' \mu(\Delta^\l)^{1-\delta'}\\
\nonumber
&\leq \sum_{|\l|=n} C'
\mu(\Delta^\l)^{1-\alpha_4} e^{-C n (\alpha_4-\delta')} \leq C' C_4^n e^{-C
n
(\alpha_4-\delta')} \leq C'.
\end{align}
The result follows.
\end{proof}

The following simple result will allow us to control the set of escaping
points as well.

\begin{theorem} \label{thm:nr}

Assume that $T$ is fast decaying.  Let $\Delta^n \subset \Delta$ be the
domain of $T^n$ and let $\Delta^\infty=\bigcap_{n \in \N} \Delta^n$.  Then
$\HD(\Delta \setminus \Delta^\infty) \leq p-1-\frac {\alpha_1}
{1+\alpha_1}$, where $\alpha_1$ is the fast decay constant of $T$.

\end{theorem}

\begin{proof}

Note that $\Delta^n \setminus \Delta^{n+1}=T^{-n}(\Delta \setminus
\Delta^1)$, so $\HD(\Delta \setminus \Delta^\infty)=\HD(\Delta \setminus
\Delta^1)$.

For simplicity, let us map $\Delta$ to the interior of the
cube $W=[0,1]^{p-1}$ by a bi-Lipschitz map $P$.  For $M \in \N$, let us
partition $W$ into $2^{M (p-1)}$ cubes of side $\delta=2^{-M}$ in the
natural way. Let us estimate the number $N$ of
cubes that are not contained in $P(\Delta^1)$.  In order to do this, we
estimate the total volume $L$ of those cubes.

For fixed $\epsilon>0$, $L$ is at most the sum $L_0$ of the
volumes of all $P(\Delta^{(l)})$ with volume at most $\epsilon$, plus the
sum $L_1$ of the volumes of the $(\sqrt {p-1}) \delta$-neighborhood of the
boundary of each $P(\Delta^{(l)})$ with volume at least $\epsilon$.

By the fast decay of $T$, we obviously have $L_0 \leq C
\epsilon^{\alpha_1}$.  On the other hand, the volume of the
$(\sqrt {p-1}) \delta$-neighborhood of the boundary of each
$P(\Delta^{(l)})$ is at most
$C \delta$.  Thus $L \leq C (\delta \epsilon^{-1}+\epsilon^{\alpha_1})$. 
Taking $\epsilon=\delta^{\frac {1} {1+\alpha_1}}$, we get $L \leq 2 C
\delta^{\frac {\alpha_1} {1+\alpha_1}}$ and hence $N \leq 2 C
\delta^{-M+\frac {\alpha_1} {1+\alpha_1}}$.  The result follows.
\end{proof}

% 6 - Anomalous Lyapunov exponents (upper and lower bound).

\section{Construction of directions with non-trivial eigenfunctions}
In this section, we provide a general construction of directional flows
with non-trivial eigenfunctions in a Veech surface. This construction
makes use of very particular elements in the Veech group called \emph{Salem}.
The presence of a single element will allow us to apply a
somewhat more general geometric criterion for positivity of Hausdorff
dimension of certain exceptional Oseledets behavior, which we now
describe in the setting of locally constant cocycles.

\subsection{Lower bound on Hausdorff dimension} \label{section:stable_HD}

Let $H$ be a finite
dimensional (real or complex) vector space.
We consider locally constant $\SL(H)$-cocycles
$(T,A)$ where $T:\bigcup_{l \in \Z} \Delta^{(l)} \to \Delta$ restricts to
projective maps $\Delta^{(l)} \to \Delta$ between simplices in $\P\R^p$, $p
\geq 2$.
We will assume that there exists some $\l \in \Omega$ such that $\Delta^\l$
is compactly contained in $\Delta$, but we will not need to assume
that $T$ has bounded distortion or even that $\bigcup_{l \in \Z}
\Delta^{(l)}$ has full measure in $\Delta$.
%(see Section~\ref{section:anomalous}).

%A submonoid $M$ of $\SL(d,\Z)$ is said to have the contraction property if
%for every $v \in \R^d \setminus \{0\}$ there exists $g \in M$ such that $\|g
%\cdot v\|<\|v\|$.

\begin{theorem}

Let $(T,A)$ be a cocycle as above.  Assume that
for every $v \in H \setminus \{0\}$,
there exists $\l \in \Omega$ such that $\|A^\l \cdot v\|<\|v\|$.  Then there
exists a finite subset $J \subset \Z$ such that for
every $v \in H \setminus \{0\}$, there exists a compact set $K_v \subset
\Delta$ with positive Hausdorff dimension such that for every $x \in K_v$ we
have $T^n(x) \in \bigcup_{j \in J} \Delta^{(j)}$, $n \geq 0$,
and $\limsup \frac {1} {n} \ln \|A_n(x) \cdot v\|<0$.

\end{theorem}

\begin{proof}

Fix two words $\l',\l'' \in \Omega$ such that
$\Delta^{\l'}$ and $\Delta^{\l''}$ have disjoint closures contained
in $\Delta$. Two such words exist. Indeed, there is a word $\l_0$ so that
$\Delta^{\l_0}$ is compactly contained in $\Delta$. We can then take $\l'=(i,\l_0)$ and $\l''=(j,\l_0)$ for
two distinct integers $i$ and $j$.

By compactness, there exists $\epsilon>0$ and
a finite subset $F \subset \Omega$ such that
for every $v \in H \setminus \{0\}$,
there exists $\l(v) \in F$ such that $\|A^{\l(v)} \cdot v\|<e^{-\epsilon}
\|v\|$.  Let $J \subset \Z$ be a finite subset containing all entries of
words in $F$, as well as all entries of $\l'$ and $\l''$.

Let $F^n \subset \Omega$ be the subset consisting of the concatenation of
$n$ words (not necessarily distinct) in $F$.  By induction, we see that
for every $v \in H \setminus \{0\}$, there exists $\l^n(v) \in F^n$
such that $\|A^{\l^n(v)} \cdot v\|<e^{-n \epsilon} v$ (just take
$\l^1(v)=\l(v)$ and for $n \geq 2$ take $\l^n(v)$
as the concatenation of $\l(v)$ and $\l^{n-1}(A^{\l(v)} \cdot v)$).

Choose $n$ such that $e^{-n \epsilon}<\frac {1} {2} \max
\{\|A^{\l'}\|,\|A^{\l''}\|\}$.

For $k \geq 1$ and a sequence $(t_0,...,t_{k-1}) \in
\{0,1\}^k$, let us define
a word $\l(v,t)$ as follows.  For $k=1$,
we let $\l(v,t)=\l^n(v) \l'$ if $t=(0)$ and $\l(v,t)=\l^n(v) \l''$ if
$t=(1)$.  For $k \geq 2$ and $t=(t_0,...,t_{k-1})$, denoting $\sigma(t)=
(t_1,...,t_{k-1})$, we let
$\l(v,t)=\l(v,t_0) \l(A^{\l(v,t_0)} \cdot v,\sigma(t))$.

Recall that the simplex $\Delta$ can be endowed with its Hilbert metric.
It has the property that for any matrix $P$ so that $P \Delta$ is compactly contained in
$\Delta$, the map $P: \Delta \rightarrow \Delta$ is a contraction for the Hilbert metric.

Note that the diameter of $\Delta^{\l(v,t)}$ in the Hilbert metric of
$\Delta$ is exponentially small in $k$: indeed, the
diameter of $\Delta^{\l(v,t)}$ in $\Delta^{\l(v,t_0)}$ is equal to the
diameter of $\Delta^{\l(v,\sigma(t))}$ in $\Delta$, and the Hilbert metric
of $\Delta^{\l(v,(t_0))}$ is strictly stronger than the Hilbert metric of
$\Delta$.  Thus given an infinite sequence $t \in \{0,1\}^\N$, the sequence
$\Delta^{\l(v,(t_0,...,t_{k-1}))}$ decreases to a point denoted by
$\gamma_v(t) \in \Delta$.  The map $\gamma_v$ then provides a homeomorphism
between $\{0,1\}^\N$ and a Cantor set $K_v \subset \Delta$.

By definition, if $t \in \{0,1\}^k$ then
$\|A^{\l(v,t)} \cdot v\|<2^{-k} \|v\|$.  It thus follows that for
$x \in K_v$ we
have $\limsup \frac {1} {n} \ln \|A_n(x) \cdot v\| \leq -\frac {\ln 2} {M}$,
where $M$ is the maximal length of all possible words $\l(v,t)$, $v \in \R^d
\setminus \{0\}$, $t \in \{0,1\}$.

Let us endow $\{0,1\}^\N$ with the usual $2$-adic metric
$d_2$, where for $t \neq t'$ we let
$d_2(t,t')=2^{-k}$ where $k$ is maximal such that $t_j=t'_j$ for $j<k$.
With respect to this metric, $\{0,1\}^\N$ has Hausdorff dimension $1$.
To conclude, it is enough to show that $\gamma_v^{-1}:K \to \{0,1\}$ is
$\alpha$-H\"older for some $\alpha>0$, as this will imply that the Hausdorff
dimension of $K$ is at least $\alpha$.

Let $d$ be the spherical metric on $\P H$.  Let $\epsilon_0>0$ be such
that for every $x \in \partial \Delta$ and
$y \in \bigcup_{v \in H \setminus
\{0\}} \bigcup_{t \in
\{0,1\}} \Delta^{\l(v,t)}$
we have $d(x,y)>\epsilon_0$. Such $\epsilon_0$ exists since all $\Delta^{\l(v,t)}$ are contained
in $\Delta^{\l'} \cup \Delta^{\l''}$ which is compactly contained in $\Delta$.
Let $\Lambda>1$ be an upper bound on the
derivative of the projective actions of any $A^{\l(v,t)}$, $v \in \R^d
\setminus \{0\}$, $t \in \{0,1\}$.  For $k \in \N$, and $t \in \{0,1\}^\N$,
$\gamma_v(t)$ is contained in $\Delta^{\l(v,(t_0,...,t_k))}$ and hence
at distance at least $\epsilon \Lambda^{-k}$ from $\partial
\Delta^{\l(v,(t_0,...,t_{k-1}))}$.  It follows that if
$d_2(t,t') \geq 2^{1-k}$ then
$d(\gamma_v(t),\gamma_v(t')) \geq \epsilon_0 \Lambda^{-k}$.  The
result then follows with $\alpha=\frac {\ln 2} {\ln \Lambda}$.
\end{proof}

The previous result would have been
enough to construct continuous eigenfunctions.
In order to construct discontinuous eigenfunctions
as well, we will need the following more precise result.

\begin{theorem} \label{thm:HD_pos_with_control_of_convergence}
Let $(T,A)$ be a cocycle as above.
Assume that for every $v \in H \setminus \{0\}$,
there exist $\l,\tilde \l \in \Omega$ such that $\|A^\l \cdot
v\|<\|v\|<\|A^{\tilde \l} \cdot v\|$.  Then there
exists a finite subset $J \subset \Z$ such that for
every $v \in H \setminus \{0\}$, and for every sequence $a_k \in \R_+$,
$k \in \N$, such that $\sup_k |\ln a_k-\ln a_{k+1}|<\infty$,
there exists a compact set $K_v \subset
\Delta$ with positive Hausdorff dimension such that for every $x \in K_v$ we
have $T^n(x) \in \bigcup_{l \in J} \Delta^{(l)}$, $n \geq 0$, and there
exists a strictly increasing subsequence $m_k$, $k \in \N$, such that
$\sup_k m_{k+1}-m_k<\infty$ and $\sup_k |\ln \|A_{m_k}(x) \cdot v\|-\ln a_k|
<\infty$.
\end{theorem}

\begin{proof}

Fix two words $\l',\l'' \in \Omega$ such that
$\Delta^{\l'}$ and $\Delta^{\l''}$ have disjoint closures contained
in $\Delta$.

Let $C_0$ be an upper bound for $|\ln a_j-\ln a_{j+1}|$.

As in the proof of the previous theorem, define a finite set $F
\subset \Omega$ such that for every $v \in \R^d \setminus \{0\}$, there
exist $\l^c(v),\l^e(v) \in F$, such that
\[
\max_{\l \in \{\l',\l''}
\|A^{\l^c(v) \l} \cdot v\|<e^{-C_0} \|v\|,
\]
\[
\min_{\l \in \{\l',\l''}
\|A^{\l^e(v) \l} \cdot v\|>e^{C_0} \|v\|.
\]

Given $k \geq 1$ and a sequence $t=(t_0,...,t_{k-1})
\in \{0,1\}^k$, define $\l(v,t)$ by
induction as follows.  If $k=1$, then we let $\l(v,t)=\l^a \l^b$ where
$\l^a=\l^c(v)$ if $\|v\|>a_0$, $\l^a=\l^e(v)$ if $\|v\| \leq a_0$,
$\l^b=\l'$ if $t=0$ and $\l^b=\l''$ if $t=1$.  If $k \geq 2$, we let
$\l(v,t)=\l(v,t_0) \l(A^{\l(v,t_0)} \cdot v,\sigma(t))$, where
$\sigma(t_0,...,t_{k-1})=(t_1,...,t_{k-1})$.

Notice that the set $G \subset \Omega$
of possible words $\l(v,t)$ with $v \in \R^d \setminus
\{0\}$ and $t \in \{0,1\}$ is finite.

By induction, we get $|\ln \|A^{\l(v,t)} \cdot v\|-\ln a_k| \leq |\ln
\|v\|-\ln a_0|+C_1$, where
\[
C_1=\max_{\l \in G}
\{\ln \|A^\l\|,\ln \|(A^\l)^{-1}\|\}.
\]

As in the proof of the previous theorem, we
define $\gamma_v:\{0,1\}^\N \to \Delta$ so that $\gamma_v(t)$ is the
intersection of the $\Delta^{\l(v,(t_0,...,t_{k-1}))}$, and conclude that
$K_v=\gamma_v(\{0,1\}^\N)$ is a Cantor set of positive Hausdorff dimension.

By construction, if $x=\gamma_v(t)$, then for every $n \in \N$
we have $A_n(x) \in \Delta^{(j)}$ for some entry $j$ of some word in $G$. 
Moreover,
$|\ln \|A_{m_k}(x) \cdot v\|-\ln a_k| \leq |\ln \|v\|-\ln a_0|+C_1$ where
$m_k$ is the length of $\l(v,(t_0,...,t_{k-1}))$.  In particular, $m_k$ is
strictly increasing and $m_{k+1}-m_k$ is bounded by the maximal length of
the words in $G$.
\end{proof}

\subsection{Salem elements and eigenfunctions}

A real number $\lambda$ is a \emph{Salem} number if it is
an algebraic integer greater than 1, all its conjugates have absolute
values not greater than 1 and at least one has absolute value 1.
These conditions imply that the minimal polynomial of a
Salem number is reciprocal and that all conjugates have modulus one
except $\lambda$ and $1/\lambda$. For $M \in \SL(2,\R)$, we say that $M$
is a \emph{Salem} matrix if its dominant eigenvalue is a Salem number.

Let $(S,\Sigma,\omega)$ be a Veech surface,
$\Gamma$ its Veech group and $k$ the trace
field of $\Gamma$. We recall that the action of the Veech group
on the tautological subspace $V = \R \Re(\omega) \oplus \R \Im(\omega)$
is naturally identified with the Veech group (see Section~\ref{section:Veech_surfaces}).
For each $\sigma \in \Gal(k / \Q)$ there is a well defined conjugate
$V^\sigma$ of $V$ which is preserved by the affine group of
$(S,\Sigma,\omega)$. These actions identifies to conjugates of the
Veech group (see Section~\ref{section:Galois_conjugate}). Salem elements
in Veech group have an alternative definition: an element of a Veech group
is Salem if and only if it is direct hyperbolic\footnote{An hyperbolic
matrix in $\SL(2,\R)$ can either have a positive or negative dominant
eigenvalue. The former case is called \emph{direct hyperbolic}.} and its Galois conjugates
are elliptic.

\begin{theorem}~\label{thm:non_wm_expl}
  Let $(S,\Sigma,\omega)$ be a
non arithmetic Veech surface and assume that
its Veech group contains a Salem element. Then
  \begin{enumerate}
	\item the set of angles whose directional flow has a continuous
eigenfunction has positive Hausdorff dimension;
	\item the set of angles whose directional flow has a measurable
discontinuous eigenfunction has positive Hausdorff dimension.
  \end{enumerate}
\end{theorem}

To build directions with eigenvalues,
we use a criterion proved in~\cite{BressaudDurandMaass} which is a
partial converse of the Veech criterion
(see Section~\ref{section:veech_criterion}). An earlier version of this
criterion appears in the paper of Veech~\cite{Veech1984}.
The criterion of~\cite{BressaudDurandMaass} only concerns
\emph{linearly recurrent systems}: the translation flow of
$(S,\omega)$ is linearly recurrent if there exists a
constant $K$ such that for any horizontal interval $J$
embedded in $(S,\omega)$ the maximum return time to $J$ is
bounded by $K/|J|$. Equivalently, a translation surface is
linearly recurrent if and only if the associated forward Teichm\"uller geodesic
is bounded in the moduli space of translation surfaces.
For these equivalence we refer to the so-called
\emph{Vorobet's identity} in the paper of Marchese-Hubert-Ulcigrai~\cite{MarcheseHubertUlcigrai}.

\begin{theorem}[\cite{BressaudDurandMaass}] \label{thm:eig_lin_rec}
Let $U$ be a relatively compact open subset in the moduli space
$\MM_g(\kappa)$ in which the Hodge bundle admits a trivialization
and let $A_n$ be the associated Kontsevich Zorich cocycle.
Let $(S,\Sigma,\omega) \in U$ be such that the return times to $U$ have
bounded gaps, then
\begin{enumerate}
\item $\nu$ is a continuous eigenvalue of $(S,\omega)$ if and only
if there exists an integer vector $v \in H^1(S;\Z) \backslash \{0\}$
such that
\[
\sum_{n \geq 0} \|A_n(\omega) \cdot (\nu \Im(\omega) - v)\| < \infty.
\]
\item $\nu$ is an $L^2$ eigenvalue of $(S,\Sigma,\omega)$ if and only if
there exists an integer vector $v \in H^1(S;\Z) \backslash \{0\}$ such that
\[
\sum_{n \geq 0} \|A_n(\omega) \cdot (\nu \Im(\omega) - v)\|^2 < \infty.
\]
\end{enumerate}
\end{theorem}
Actually, the criterion applies to the Cantor space obtained from the
translation surface where each point that belongs to a singular leaf is
doubled. The continuity in that space is weaker than the continuity on
the surface. For a continuous eigenfunction $f$, the cohomological equation $f(\phi_T(x)) = e^{2i\pi \nu T}f(x)$
allows to recover the value along a leaf from the value at one point on that leaf.
But singular leaves either coincide in the past or in the future. Hence the values of the eigenfunctions
on the doubled leaves must be identical. In other words, the function is well defined on the quotient
which is the surface.

\begin{rem} \label{rmk:nb_eig_lin_rec}
  Theorem~\ref{thm:eig_lin_rec} allows to strengthen the
conclusion of Theorem~\ref{thm:nb_eig}
  when the flow is linearly recurrent: if $(S,\Sigma,\omega)$ is a non-arithmetic Veech surface with
  trace field $k$ and whose flow is linearly recurrent then it admits either $0$ or $[k:\Q]$ rationally
  independent eigenvalues. Moreover, they are simultaneously contiuous or discontinuous.
  Indeed (following the proof of Theorem \ref {thm:nb_eig}),
  any two non-zero
  ``potential eigenvalues'' $\mu,\mu' \in \R$ such that there
  exists $v,v' \in W_\Z$ with $\mu \Im(\omega)-v$ and $\mu' \Im(\omega')-v'$
  belong to $E^s$ are such that
  $\frac {\|A_n(\omega) \cdot (\mu \Im(\omega)-v)\|} {\|A_n(\omega) \cdot
  (\mu' \Im(\omega)-v')\|}$ is uniformly bounded away from zero or infinity
  (independently of $n$). Indeed, writing $v = \sum v_\sigma$ each component $v_\sigma$ belongs to $E^s(V^\sigma)$ and is non-zero. One concludes using the fact that $E^s(V^\sigma)$ is one dimensional.
  
  By Theorem \ref {thm:eig_lin_rec},
  $\mu$ is a continuous eigenvalue if and only if $\mu'$ is, and $\mu$
  is an $L^2$ eigenvalue if and only if $\mu'$ is.  Since the set $\Theta$ of
  potential eigenvalues is either $\{0\}$ or has dimension $[k:\Q]$ over
  $\Q$, the result follows.
\end{rem}

Before going into details of the proof, we provide
various examples of Veech surfaces which contain Salem elements.  In
particular the next result shows that Theorem \ref
{thm:non_wm_expl_for_quad_trace_field} follows from Theorem
\ref {thm:non_wm_expl}.

\begin{proposition}[\cite{BressaudBufetovHubert}, Proposition 1.7]
  A Veech surface with quadratic trace field has a Salem element in its
Veech group.
\end{proposition}
\begin{proof}
  We follow \cite{BressaudBufetovHubert}. The Veech group has only one
conjugate and this conjugate is non-discrete
(see Proposition~\ref{prop:conjugates_indiscrete}). On the other hand
we know from a result of Beardon (\cite{Beardon1983} Theorem~8.4.1)
that any non-discrete subgroup of $\SL(2,\R)$ contains an elliptic
element with irrational angle. If $g$ is an element of the Veech
group whose conjugate is an irrational rotation then $g$ can not be
elliptic as it is of infinite order and the Veech group is discrete and
$g$ can not be parabolic because a conjugate of a parabolic element
is again parabolic. Hence $g$ is hyperbolic and $g^2$ is a Salem
element of the Veech group.
\end{proof}

For more general Veech surfaces, we obtain examples through
computational experiments (see appendix~\ref{appendix:salem_triangle}
for explicit matrices).

\begin{proposition}
  For respectively odd $q \leq 15$ and any $q \leq 15$ the triangle groups
$\Delta(2,q,\infty)$ and $\Delta(q,\infty,\infty)$ contain Salem elements.
\end{proposition}

In particular, Theorem~\ref{thm:non_wm_expl} holds for many billiards in
regular polygons $P_n$ defined in the introduction.

Now, we proceed to the proof of Theorem~\ref{thm:non_wm_expl}.

\begin{lemma} \label{lem:irr_eig_salem}
  Let $\lambda$ be a Salem number and
$\{\lambda, 1/\lambda, e^{i\alpha_1}, e^{-i\alpha_1},
\ldots, e^{i\alpha_k}, e^{-i \alpha_k}\}$ its Galois conjugates,
then $\alpha_1, \ldots, \alpha_k, \pi$ are rationally independant.
\end{lemma}

\begin{proof}
  Let $n_1, \ldots n_k, m$ be integer such that 
  \begin{equation} \label{eq:salem_product}
  n_1 \alpha_1 + \ldots + n_k \alpha_k = 2 m \pi.
  \end{equation}
  Then
  \begin{equation} \label{eq:salem_power}
	\left( e^{i \alpha_1} \right)^{n_1} \ldots
\left( e^{i \alpha_k} \right)^{n_k} = 1.
  \end{equation}
  Each element of the Galois group is a
field homomorphism and hence it preserves the
partition $\{\lambda, 1/\lambda\}$, $\{e^{i \alpha_1}, e^{-i\alpha_1}\}$,
\ldots, $\{e^{i\alpha_k}, e^{-i \alpha_k}\}$. By definition,
the Galois group acts transitively on
 $\{\lambda, 1/\lambda, e^{i\alpha_1}, e^{-i\alpha_1}, \ldots,
e^{i\alpha_k}, e^{-i \alpha_k}\}$ and for each $i = 1, \ldots, k$
there exists a field homomorphism that maps $e^{i \alpha_j}$ to
$\lambda$ and all other $e^{i \alpha_{j'}}$ to some
$e^{\pm \alpha_{j''}}$. By applying this field
homomorphism to the equality (\ref{eq:salem_power}) and taking
absolute value we get that $n_i=0$ because $|\lambda| > 1$.  Hence the
relation (\ref{eq:salem_product}) is trivial.
\end{proof}

We now show that the presence of Salem elements allows us to verify the
hypothesis of Theorem~\ref{thm:HD_pos_with_control_of_convergence}.

\begin{lemma} \label{lem:contraction_hyp_and_salem}
Let $(S,\Sigma,\omega)$ be a Veech surface, $k$ its
holonomy field,
$V = \R \Re\omega \oplus \Im\omega$ the tautological subspace and
$W^0 = \bigoplus_{\sigma \in \Gal(k / \Q) \backslash \{\id\}} V^\sigma$.
Let $\gamma$ be a Salem element of the Veech group and $\gamma_j$, $j \geq
1$, be such that for each $\sigma \neq \id$, the norm of the conjugates
$\gamma_j^\sigma$ grows to infinity.
Denote by $g$ and $g_k$ their actions on $W^0$.  Then
for any $v \in W^0 \setminus \{0\}$ there exist positive integers $n_-,n_+$
and $k_-,k_+$ such that elements $g_-=g_{k_-} g^{n_-}$ and $g_+=g_{k_+}
g^{n_+}$ satisfy
$\|g_- v \| < \|v\| < \|g_+ v\|$.
\end{lemma}

\begin{proof}

Let $v \in W^0$ be a unit vector, and for $\sigma \neq \id$, let
$\pi_\sigma:W^0 \to V^\sigma$ be the projection on
the $V^\sigma$-coordinate.  Let $D \subset \Gal(k/\Q) \setminus \{\id\}$
be the set of all
$\sigma$ such that $\pi_\sigma(v) \neq 0$.
We show that there
exists positive integers $n$ and $k$ such that for
all $\sigma \in D$,
we have $\|g_k g^n \pi_\sigma(v)\| < \frac {\|v\|} {\# D}$,
which implies the first inequality with $n=n_-$ and $k=k_-$.
The other inequality may be obtained by the very same argument.

Let $\theta^\sigma_k$ be the norm of $\gamma^\sigma_k$.  By hypothesis,
$\theta^\sigma_k>1$ for every $\sigma$ and every $k$ sufficiently large.
Let $F^\sigma_- \in \P V^\sigma$ be the most contracted direction of
$\gamma^\sigma_k$ in $V^\sigma$.

Consider the action of $g$ on the torus
$\prod_{\sigma \in D} \P V^\sigma$.
By Lemma~\ref{lem:irr_eig_salem}, this action is minimal.  In particular
there exist a sequence of non-negative integer $n_j$ such that
$g^{n_j} \pi_\sigma(v)$ converges to a vector $w_\sigma$ in
$F^\sigma_- \setminus \{0\}$ for every $\sigma \in D$.  In particular, for
fixed $k \in \N$ we have $\lim_{j \to \infty}
\|g_k g^{n_j} \pi_\sigma(v)\|=\frac {\|w_\sigma\|} {\theta^\sigma_k}$
for every $\sigma \in D$.  The result follows by taking $k$ such that
$\frac {\|w_\sigma\|} {\theta^\sigma_k}<\frac {\|v\|} {\# D}$ (recall that
$\theta^\sigma_k \to \infty$ as $k \to \infty$).
\end{proof}

\begin{proof}[Proof of Theorem~\ref{thm:non_wm_expl}]

Let $V=\R \Re(\omega) \oplus \R \Im(\omega)$
be the tautological subspace of the cohomology
$H^1(S;\R)$ and $W = \oplus_{\sigma: k \rightarrow \R} V^\sigma$ and
$W^0=\oplus_{\sigma \neq \id} V^\sigma$.
We are going to construct a Markov model $(T:\Delta\rightarrow\Delta,A)$ for the Kontsevich-Zorich
over the $\SL(2,\R)$ orbit $\CC$ of $(S,\Sigma,\omega)$, a point $x \in
\Delta$ and a positive integer $n$, and a sequence of
points $y_k \in \Delta$ and positive integers $n_k$, such that $A_n(x)|V$ is
a Salem element and $\lim_{k \to \infty}
\inf_{\sigma \neq \id} \|A_{n_k}|V^\sigma\|=\infty$.  First, let us show how
this construction implies the result.

By Lemma \ref {lem:contraction_hyp_and_salem}, this
implies that the hypothesis of Theorem \ref
{thm:HD_pos_with_control_of_convergence} are satisfied for the cocycle
$(T,A|W^0)$, so for every
$w \in W^0 \setminus \{0\}$,
there exist subsets $Z_c,Z_m \subset \Lambda$ of positive
Hausdorff dimension such that for $u \in Z_c$ we have $\sum \|A_n(u) \cdot
w\|<\infty$ and for $u \in Z_m$ we have $\sum \|A_n(u) \cdot w\|^2<\infty$
and $\sum \|A_n(u) \cdot w\|=\infty$.  Moreover, since Theorem \ref
{thm:HD_pos_with_control_of_convergence}
provides also that $T^n(u)$ visits only finitely many distinct
$\Delta^{(l)}$, the return times $r(T^n(u))$ remain bounded so that the
forward Teichm\"uller geodesic starting at $u$ is bounded in moduli space.

Let us take $w$ as the projection on $W^0$, along $V$, of a non-zero vector
$v \in W \cap H^1(S;\Z)$.  Fix $u \in Z_c \cup Z_m$.  Let $\omega=p(u,0)$
and write $v=w+\nu \Im \omega+\eta \Re \omega$.  Then
$A_n(u) \cdot (\nu \Im \omega-v)=-A_n(u) \cdot w-\eta A_n(u) \cdot \Re
\omega$.  Note that $A_n(u) \cdot \Re \omega$ decays exponentially fast,
since $\Re \omega$ is in the direction of the strongest contracting
subbundle of the Kontsevich-Zorich cocycle.
Notice that $\nu \neq 0$,
otherwise the integer non-zero vectors $A_n(u) \cdot v$ would converge to
$0$.  By Theorem \ref {thm:eig_lin_rec}, if $u \in
Z_c$, then the vertical flow for $\omega$ admits a
continuous eigenfunction with eigenvalue $\nu$ and if $u \in Z_m$ then
the vertical flow for $\omega$ admits a measurable eigenfunction, but no
continuous eigenfunction, with eigenvalue $\nu$.

We have thus obtained positive Hausdorff dimension subsets $p(Z_c \times
\{0\})$ and $p(Z_m \times \{0\})$ of an unstable horocycle for which the
vertical flow has continuous and measurable but discontinuous
eigenfunctions.  Using Lemma \ref {redu}, we transfer the result to the
directional flow in any surface in $\CC$, giving the desired conclusion.

We now proceed with the construction of the Markov model.
Recall that for each hyperbolic element $\gamma$
in the Veech group, there exists a periodic orbit $O_\sigma$
of the Teichm\"uller
flow in the $\SL(2,\R)$ orbit and a positive integer $n_\gamma$,
such that the restriction to $V$ of the
$n_\gamma$-th iterate of the monodromy of the
Kontsevich-Zorich cocycle along this periodic orbit is conjugate to
$\gamma$.

Let $\gamma$ be a Salem element in the Veech group, and let us consider
the Markov model $(T,A)$ for the Kontsevich-Zorich cocycle obtained by
taking a small Poincar\'e section $Q$ through some $x \in O_\gamma$.
Then clearly $A_{n_\gamma}(x)|V$ is a Salem element.

On the other hand, by Lemma \ref {lem:separation_LE_Veech_surfaces}, $(T,A|V^\sigma)$ has a positive
Lyapunov exponent for every $\sigma$.  Thus
for large $n$ and for a set of $y$ of probability close to
$1$, the norm of $\|A_n(y)|V^\sigma\|$ is large.  In particular, for each $k
\in \N$ there exists $y_k$ and a positive integer $n_k$ such that
$\|A_{n_k}|V^\sigma\|>k$ for every $\sigma \neq \id$.  The result follows.
\end{proof}

% 7 - directions with non-trivial eigenfunctions

\appendix
\section{Salem elements in triangle groups} \label{appendix:salem_triangle}
In this appendix we provide explicit Salem matrices in triangle groups $\Delta(2,q,\infty)$ and $\Delta(q,\infty,\infty)$ for trace field of degree greater than two and small values of $q$. The matrices are given in terms of the standard generators $s$, $t$ of the triangle group $\Delta(p,q,r)$ that satisfy
\[
  s^p = t^q = (st)^r = \pm id.
\]
Instead of writing down the minimal polynomial of the eigenvalue, we write it for half the trace. The roots of modulus less than one are cosine of the angles of the corresponding elliptic matrices.

The array stops at the values $q=17$ for $\Delta(2,q,\infty)$ and $q=16$ for $\Delta(q,\infty,\infty)$ for which we were unable to find Salem elements. All these examples were obtained using the mathematical software Sage~\cite{Sage}.

\subsection{Salem elements in $\Delta(2,q,\infty)$}
$$\renewcommand\arraystretch{1.3}\begin{array}{|c|c|c|}
\hline 
\text{q} & \text{degree} & \text{matrix $m$} \\ \hline 
& \multicolumn{2}{|c|}{\text{minimal polynomial of trace(m)/2}} \\ 
& \multicolumn{2}{|c|}{\text{approximate conjugates of trace(m)/2}} \\ \hline \hline 
   7 & 3 & t^{3}.s\\ 
  \hline &\multicolumn{2}{|c|}{x^{3} - 2 x^{2} - x + 1}\\ 
 & \multicolumn{2}{|c|}{2.247,\ 0.5550,\ -0.8019}\\ \hline \hline 
 9 & 3 & t^{4}.s\\ 
  \hline &\multicolumn{2}{|c|}{x^{3} - 3 x^{2} + 1}\\ 
 & \multicolumn{2}{|c|}{2.879,\ 0.6527,\ -0.5321}\\ \hline \hline 
11 & 5 & t^{5}.s.t^{4}.s\\ 
  \hline &\multicolumn{2}{|c|}{x^{5} - \frac{39}{2} x^{4} - 47 x^{3} - \frac{243}{8} x^{2} - \frac{17}{16} x + \frac{89}{32}}\\ 
 & \multicolumn{2}{|c|}{21.73,\ 0.2425,\ -0.6156,\ -0.8781,\ -0.9764}\\ \hline \hline 
13 & 6 & t^{7}.s.t^{7}.s.t^{4}.s\\ 
  \hline &\multicolumn{2}{|c|}{x^{6} - 227 x^{5} - 11 x^{4} + 318 x^{3} + 41 x^{2} - 110 x - 25}\\ 
 & \multicolumn{2}{|c|}{227.0,\ 0.9072,\ 0.8412,\ -0.2464,\ -0.6697,\ -0.8746}\\ \hline \hline 
15 & 4 & t^{7}.s\\ 
  \hline &\multicolumn{2}{|c|}{x^{4} - 4 x^{3} - 4 x^{2} + x + 1}\\ 
 & \multicolumn{2}{|c|}{4.783,\ 0.5112,\ -0.5473,\ -0.7472}\\ \hline
\end{array}$$

\subsection{Salem elements in $\Delta(q,\infty,\infty)$}
$$\renewcommand\arraystretch{1.3}\begin{array}{|c|c|c|}
\hline 
\text{q} & \text{degree} & \text{matrix $m$} \\ \hline 
& \multicolumn{2}{|c|}{\text{minimal polynomial of trace(m)/2}} \\ 
& \multicolumn{2}{|c|}{\text{approximate conjugates of trace(m)/2}} \\ \hline \hline 
 7 & 3 & t.s^{3}\\ 
  \hline &\multicolumn{2}{|c|}{x^{3} - 3 x^{2} - 4 x - 1}\\ 
 & \multicolumn{2}{|c|}{4.049,\ -0.3569,\ -0.6920}\\ \hline \hline 
 8 & 4 & t.s^{2}.t.s^{3}\\ 
  \hline &\multicolumn{2}{|c|}{x^{4} - 24 x^{3} + 15 x^{2} + 4 x + \frac{1}{8}}\\ 
 & \multicolumn{2}{|c|}{23.35,\ 0.8571,\ -0.03655,\ -0.1709}\\ \hline \hline 
 9 & 3 & t.s^{2}\\ 
  \hline &\multicolumn{2}{|c|}{x^{3} - 3 x^{2} + 1}\\ 
 & \multicolumn{2}{|c|}{2.879,\ 0.6527,\ -0.5321}\\ \hline \hline 
10 & 4 & t.s^{3}.t.s^{7}\\ 
  \hline &\multicolumn{2}{|c|}{x^{4} - 49 x^{3} - \frac{441}{4} x^{2} - \frac{291}{4} x - \frac{199}{16}}\\ 
 & \multicolumn{2}{|c|}{51.18,\ -0.2644,\ -0.9504,\ -0.9672}\\ \hline \hline 
11 & 5 & t.s^{4}.t.s^{7}\\ 
  \hline &\multicolumn{2}{|c|}{x^{5} - \frac{155}{2} x^{4} - 122 x^{3} - \frac{459}{8} x^{2} - \frac{173}{16} x - \frac{23}{32}}\\ 
 & \multicolumn{2}{|c|}{79.05,\ -0.1907,\ -0.2214,\ -0.2388,\ -0.9015}\\ \hline \hline 
12 & 4 & t.s^{2}.t.s^{3}\\ 
  \hline &\multicolumn{2}{|c|}{x^{4} - 24 x^{3} - 61 x^{2} - 48 x - \frac{191}{16}}\\ 
 & \multicolumn{2}{|c|}{26.38,\ -0.5254,\ -0.9096,\ -0.9468}\\ \hline \hline 
13 & 6 & t.s^{4}.t.s^{5}.t^{-1}.s^{4}.t^{-1}.s^{5}\\ 
  \hline &\multicolumn{2}{|c|}{x^{6} - \frac{43107}{2} x^{5} - \frac{188297}{4} x^{4} - 26514 x^{3} + \frac{53979}{8} x^{2} + \frac{304515}{32} x + \frac{124175}{64}}\\ 
 & \multicolumn{2}{|c|}{21560.,\ 0.5373,\ -0.3375,\ -0.7022,\ -0.8374,\ -0.8440}\\ \hline \hline 
14 & 6 & t.s^{5}.t.s^{9}\\ 
  \hline &\multicolumn{2}{|c|}{x^{6} - 125 x^{5} - \frac{955}{4} x^{4} - \frac{45}{4} x^{3} + \frac{1653}{8} x^{2} + \frac{967}{8} x + \frac{1009}{64}}\\ 
 & \multicolumn{2}{|c|}{126.9,\ 0.9692,\ -0.1912,\ -0.6930,\ -0.9794,\ -0.9879}\\ \hline \hline 
15 & 4 & t.s^{3}\\ 
  \hline &\multicolumn{2}{|c|}{x^{4} - 4 x^{3} - 4 x^{2} + x + 1}\\ 
 & \multicolumn{2}{|c|}{4.783,\ 0.5112,\ -0.5473,\ -0.7472}\\ \hline 
\end{array}$$


\begin{thebibliography}{KMS86}

  \bibitem[AD]{AD2} Artur Avila, Vincent Delecroix,
  \textit {Large deviations for algebraic $\SL(2,\R)$-invariant measures in moduli space},
  in progress.

  \bibitem[AF07]{AvilaForni2007} Artur Avila, Giovanni Forni,
  \textit{Weak mixing for interval exchange transformations and translation flows},
  Ann. of Math. \textbf{165} (2007) 637--664.

  \bibitem[AF]{AF2} Artur Avila, Giovanni Forni,
  \textit{Weak mixing in L shaped billiards},
  in progress.

%  \bibitem[AGY06]{AvilaGouezelYoccoz2006} Artur Avila, S\'ebastien Gou\"ezel, Jean-Cristophe Yoccoz, \textit{Exponential mixing for the Teichmüller flow}, 
%   Publications Math\'ematiques de l'IH\'ES 104 (2006), 143-211.

  \bibitem[AV07]{AV} Artur Avila, Marcelo Viana,
  \textit{Simplicity of Lyapunov spectra: proof of the Zorich-Kontsevich conjecture},
  Acta Mathematica \textbf{198} (2007) 10--56.

  \bibitem[Be83]{Beardon1983} Alan F. Beardon,
  \textit{The geometry of discrete groups},
  Graduate Texts in Mathematics. Springer-Verlag. New-York (1983).

  \bibitem[BM10]{BouwMoeller2010} Irene Bouw, Martin M\"oller,
  \textit{Teichm\"uller curves, triangle groups, and Lyapunov exponents},
  Annals of Math. \textbf{172} (2010) 139--185.

 \bibitem[BBH]{BressaudBufetovHubert} X. Bressaud, A. Bufetov, P. Hubert,
 \textit{Deviation of ergodic averages for substitution dynamical systems with eigenvalues of modulus one},
  Proc. Lond. Math. Soc. \textbf{109} no. 2 (2014) 483--522.

  \bibitem[BDM1]{BressaudDurandMaass} Xavier Bressaud, Fabien Durand, Alejandro Maass,
  \textit{Necessary and sufficient conditions to be an eigenvalue for linearly recurrent dynamical Cantor systems},
  J. of the London Math. Soc. \textbf{72} no. 3 (2005) 799--816.

  \bibitem[BDM2]{BDM2} Xavier Bressaud, Fabien Durand, Alejandro Maass,
  \textit {Eigenvalues of finite rank Bratteli-Vershik dynamical systems},
  Ergodic Theory Dynam. Systems \textbf{30} no. 3 (2010) 639--664.

  \bibitem[Ca04]{Calta2004} Kariane Calta,
  \textit{Veech surfaces and complete periodicity in genus two},
  J. Amer. Math. Soc. \textbf{24} (2004) 871--908.

  \bibitem[CE]{ChaikaEskin} Jon Chaika, Alex Eskin,
  \textit{Every flat surface is Birkhoff and Osceledets generic in almost every direction},
  J. of Mod. Dyn. \textbf{9} (2015) 1--23.

  \bibitem[CS08]{CaltaSmillie2008} Kariane Calta, John Smillie,
  \textit{Algebraically periodic translation surfaces},
  J. Mod. Dyn. \textbf{2} (2008) 209--248.

  \bibitem[EM]{EskinMatheus} Alex Eskin, Carlos Math\'eus,
  \textit{A coding-free simplicity criterion for the Lyapunov exponents of Teichmueller curves},
  to appear in Geometriae Dedicata.

  \bibitem[Fo02]{Forni2002} G. Forni,
  \textit{Deviation of ergodic averages for area-preserving flows on surfaces of higher genus},
  Ann. of Math. \textbf{155} no. 1 (2002) 1--103.

  \bibitem[GP]{GP} M. Guenais, F. Parreau,
  \textit {Valeurs propres des transformations li\'ees aux rotations irrationnelles et aux fonctions en escalier},
  preprint arXiv 0605250v1 (2006).

  \bibitem[GJ00]{GutkinJudge2000} E. Gutkin, C. Judge,
  \textit{Affine mappings of translation surfaces: geometry and arithmetic},
  Duke Math. J. \textbf{103} no. 2 (2000) 191--213.
 
  \bibitem[Ho12]{Hooper2012} W. Patrick Hooper,
  \textit{Grid graphs and lattice surfaces},
  Int. Math. Res. Notices (2012).

  \bibitem[MHU]{MarcheseHubertUlcigrai} Luca Marchese, Pascal Hubert, Corinna Ulcigrai,
  \textit{Larange spectra in Teichm\"uller dynamics via renormalization},
  Geom. Funct. Anal. \textbf{25} (2015) 180--255.

  \bibitem[HS06]{HubertSchmidt2006} P. Hubert, T. Schmidt,
  \textit{An introduction to Veech surfaces}, 
  Handbook of dynamical systems  Vol. 1B
  Elsevier B. V., Amsterdam, (2006), 501--526.

  \bibitem[Ka80]{Katok1980} Anatole Katok,
  \textit{Interval exchange transformations and some special flows are not mixing},
  Israel Journal of Math. \textbf {35} (1980) 301--310.

  \bibitem[Ke75]{Keane1975} Michael Keane,
  \textit{Interval exchange transformations},
  Math. Z. \textbf{141} (1975) 25--31.

  \bibitem[KS00]{KenyonSmillie2000} R. Kenyon, J. Smillie,
  \textit{Billiards on rational angled triangles},
  Comment. Math. Helv. \textbf{75} no. 1 (2000) 65--108.

  \bibitem[KMS86]{KerckhoffMasurSmillie1986} Steven Kerckhoff, Howard Masur, John Smille,
  \textit{Ergodicity of billiard flows and quadratic differentials},
  Ann. of Math. \textbf{124} (1986) 293--311.

  \bibitem[LH06]{LanneauHubert2006} Pascal Hubert, Erwan Lanneau,
  \textit{Veech groups without parabolic elements},
  Duke Math. J. \textbf{133} no. 2 (2006) 335--346.

  \bibitem[LN14]{LanneauNguyen2014} Erwan Lanneau, Duc-Manh Nguyen,
  \textit{Teichmueller curves generated by Weierstrass Prym eigenforms in genus three and genus four},
  J. Topol. \textbf{7} no. 2 (2014) 475--522. 

  \bibitem[Ma82]{Masur1982} Howard Masur,
  \textit{Interval exchange transformations and measured foliations},
  Ann. of Math. (2) \textbf{115} no. 1 (1982) 169--200. 

  \bibitem[Ma92]{Masur1992} Howard Masur,
  \textit{Hausdorff dimension of the set of nonergodic foliations of a quadratic differential},
  Duke Math. J. \textbf{66} no. 3 (1992) 387--442. 

  \bibitem[McM03]{McMullen2003}, Curtis T. McMullen,
  \textit{Billiards and Teichm\"uller curves on Hilbert modular surfaces},
  J. Amer. Math. Soc. \textbf{16} (2003) 857--885.

  \bibitem[McM06]{McMullen2006} Curtis T. McMullen,
  \textit{Prym varieties and Teichmüller curves},
  Duke Math. J. \textbf{133} 3 (2006) 569--590.

  \bibitem[McM07]{McMullen2007} Curtis T. McMullen,
  \textit{Dynamics of SL2(R) over moduli space in genus two},
  Ann. of Math. (2) \textbf{165} no. 2 (2007) 397--456. 

  \bibitem[MMY05]{MarmiMoussaYoccoz2005} Stefano Marmi, Pierre Moussa, Jean-Cristophe Yoccoz,
  \textit{The cohomological equation for Roth type interval exchange transformations},
  J. Amer. Math. Soc. \textbf{18} (2005) 823--872.

  \bibitem[MT02]{MasurTabachnikov2002} Howard Masur, Serge Tabachnikov,
  \textit{Rational billiards and flat structures},
  Handbook of dynamical systems, {V}ol. 1{A}, Elsevier (2002)1015--1089.

  \bibitem[Mo06]{Moeller2006} Martin M\"oller,
  \textit{Variations of Hodge structure of Teichmueller curves}
  Journal of the AMS \textbf{19} (2006) 327--344.

  \bibitem[Sa]{Sage} {W}illiam {A}. {S}tein et al.
  \textit{ {S}age {M}athematics {S}oftware (Version 5.0)},
  The {S}age {D}evelopment {T}eam, \url{http://www.sagemath.org}.

  \bibitem[SW10]{SmillieWeiss2010} John Smillie, Barak Weiss,
  \textit{Characterizations of lattice surfaces},
  Inv. Math. \textbf{180} no. 3 (2010) 535--557

  \bibitem[Ve82]{Veech1982} William A. Veech,
  \textit{Gauss measures for transformations on the space of interval exchange maps},
  Ann. of Math. (2) \textbf{115} no. 1 (1982) 201--242. 

  \bibitem[Ve84]{Veech1984} William A. Veech,
  \textit{The metric theory of interval exchange transformations. I. Generic spectral properties},
  Amer. J. Math. \textbf{106} no. 6 (1984) 1331--1359.

  \bibitem[Ve89]{Veech1989} William A. Veech,
  \textit{Teichm\"uller curves in moduli space, Eisenstein series and application to triangular billiards},
  Inv. Math. \textbf{97} (1989) 553--583.

  \bibitem[Wa98]{Ward1998} Clayton C. Ward,
  \textit{Calculation of Fuchsian groups associated to billiards in a rational triangle}
  Ergodic Theory Dynam. Systems \textbf{18} no. 4 (1998) 1019--1042. 

  \bibitem[Wr13]{Wright2013} Alex Wright,
  \textit{Schwarz triangle mappings and Teichm\"uller curves: the Veech-Ward-Bouw-M\"oller curves},
  Geom. Funct. Anal. \textbf{23} (2013) 776--809.

\end{thebibliography}
\end{document}